\documentclass[oneside,english]{amsart}
\usepackage{lmodern}
\usepackage[T1]{fontenc}
\usepackage[latin9]{inputenc}
\usepackage{color}
\usepackage{verbatim}
\usepackage{amsthm}
\usepackage{amssymb}
\PassOptionsToPackage{normalem}{ulem}
\usepackage{ulem}

\makeatletter
\numberwithin{equation}{section}
\numberwithin{figure}{section}
\theoremstyle{plain}
\newtheorem{thm}{\protect\theoremname}[section]
  \theoremstyle{plain}
  \newtheorem{fact}[thm]{\protect\factname}
  \theoremstyle{remark}
  \newtheorem{claim}[thm]{\protect\claimname}
  \theoremstyle{definition}
  \newtheorem{defn}[thm]{\protect\definitionname}
  \theoremstyle{plain}
  \newtheorem{cor}[thm]{\protect\corollaryname}
  \theoremstyle{remark}
  \newtheorem{rem}[thm]{\protect\remarkname}
  \theoremstyle{plain}
  \newtheorem{lem}[thm]{\protect\lemmaname}
  \theoremstyle{definition}
  \newtheorem{example}[thm]{\protect\examplename}
 \theoremstyle{definition}
 \newtheorem*{defn*}{\protect\definitionname}
  \theoremstyle{remark}
  \newtheorem*{claim*}{\protect\claimname}
\theoremstyle{plain}
\newtheorem{obs}[thm]{Observation}
  \theoremstyle{plain}
  \newtheorem{prop}[thm]{\protect\propositionname}
  \theoremstyle{definition}
  \newtheorem{problem}[thm]{\protect\problemname}

\usepackage{multicol}
\usepackage{a4wide}
\usepackage{url}
\newcommand{\NTP}{NTP$_{2}$ }

\makeatother

\usepackage{babel}
  \providecommand{\claimname}{Claim}
  \providecommand{\corollaryname}{Corollary}
  \providecommand{\definitionname}{Definition}
  \providecommand{\examplename}{Example}
  \providecommand{\factname}{Fact}
  \providecommand{\lemmaname}{Lemma}
  \providecommand{\problemname}{Problem}
  \providecommand{\propositionname}{Proposition}
  \providecommand{\remarkname}{Remark}
\providecommand{\theoremname}{Theorem}

\begin{document}
\begin{flushleft}
\global\long\def\NTPT{\operatorname{NTP}_{\operatorname{2}}}
\global\long\def\C{\mathfrak{C}}
\global\long\def\Aut{\operatorname{Aut}}
\global\long\def\tp{\operatorname{tp}}
\global\long\def\id{\operatorname{id}}
\global\long\def\ist{\operatorname{ist}}
\global\long\def\C{\mathfrak{C}}
\global\long\def\alt{\operatorname{alt}}
\global\long\def\st{\operatorname{st}}
\global\long\def\dom{\operatorname{dom}}
\global\long\def\acl{\operatorname{acl}}
\global\long\def\eq{\operatorname{eq}}

\par\end{flushleft}

\def\Ind#1#2{#1\setbox0=\hbox{$#1x$}\kern\wd0\hbox to 0pt{\hss$#1\mid$\hss} \lower.9\ht0\hbox to 0pt{\hss$#1\smile$\hss}\kern\wd0} 
\def\Notind#1#2{#1\setbox0=\hbox{$#1x$}\kern\wd0\hbox to 0pt{\mathchardef \nn="3236\hss$#1\nn$\kern1.4\wd0\hss}\hbox to 0pt{\hss$#1\mid$\hss}\lower.9\ht0 \hbox to 0pt{\hss$#1\smile$\hss}\kern\wd0} 

  \theoremstyle{definition}
  \newtheorem{defClaim}[thm]{Definition/Claim}

\begin{flushleft}
\global\long\def\ind{\mathop{\mathpalette\Ind{}}}
 \global\long\def\nind{\mathop{\mathpalette\Notind{}}}
\global\long\def\leftexp#1#2{{\vphantom{#2}}^{#1}{#2}}
\global\long\def\mplus{\left\lfloor \left(m+1\right)/2\right\rfloor }

\par\end{flushleft}

\global\long\def\Aut{\operatorname{Aut}}
\global\long\def\calB{\mathcal{B}}
\global\long\def\calI{\mathcal{I}}
\global\long\def\ball#1#2{\calB(#1, #2)}
\global\long\def\calU{\mathcal{U}}
\global\long\def\calV{\mathcal{V}}
\global\long\def\Cantorspace{\functions{\omega}{2}}
\global\long\def\closure#1{\overline{#1}}
\global\long\def\concatenation{\mathrel{\smallfrown}}
\global\long\def\constantsequence#1#2{\left(#1\right){}^{#2}}
\global\long\def\diameter#1{\mathrm{diam}\left(#1\right)}
\global\long\def\Ezero{\mathbb{E}_{0}}
\global\long\def\from{\colon}
\global\long\def\Fsigma{F_{\sigma}}
\global\long\def\functions#1#2{#2^{#1}}
\global\long\def\heightcorrection#1{\raisebox{0pt}[0pt][0pt]{#1}}
\global\long\def\ideal#1{\calI_{#1}}
\global\long\def\image#1#2{#1\left[#2\right]}
\global\long\def\inverse#1{#1^{-1}}
\global\long\def\mathand{\text{ and }}
\global\long\def\N{\mathbb{N}}
\global\long\def\orbitequivalencerelation#1#2{E_{#1}^{#2}}
\global\long\def\pair#1#2{\left(#1,#2\right)}
\global\long\def\relationpower#1#2{#1^{\left(#2\right)}}
\global\long\def\restriction#1#2{#1 \upharpoonright#2}
\global\long\def\setcomplement#1{\twiddle#1}
\global\long\def\sets#1#2{\left[#2\right]{}^{#1}}
\global\long\def\suchthat{\mid}
\global\long\def\twiddle{\raisebox{1pt}{\scalebox{.75}{$\mathord{\sim}$}}}
\global\long\def\C{\mathfrak{C}}
\global\long\def\acl{\operatorname{acl}}
\global\long\def\tp{\operatorname{tp}}
\global\long\def\qf{\operatorname{qf}}
\global\long\def\Nn{\mathbb{N}}
\global\long\def\id{\operatorname{id}}
\global\long\def\SS{\mathcal{P}}
\global\long\def\EM{\operatorname{EM}}
\global\long\def\dcl{\operatorname{dcl}}
\global\long\def\Autf{\operatorname{Aut}f_{L}}
\global\long\def\eq{\operatorname{eq}}
\global\long\def\image{\mbox{image}}

\global\long\def\pamod#1{\pmod#1}

\global\long\def\nf{\mbox{nf}}
\global\long\def\Uu{\mathcal{U}}
\global\long\def\dom{\operatorname{dom}}
\global\long\def\concat{\smallfrown}
\global\long\def\Nn{\mathbb{N}}
\global\long\def\mathrela#1{\mathrel{#1}}
\global\long\def\twiddle{\mathord{\sim}}
\global\long\def\stab{\operatorname{stab}}
 \global\long\def\x{\times}
\global\long\def\diam{\operatorname{diam}}
\global\long\def\EZero{\mathbb{E}_{0}}
\global\long\def\sequence#1#2{\left\langle #1\left|\,#2\right.\right\rangle }
\global\long\def\set#1#2{\left\{  #1\left|\,#2\right.\right\}  }
\global\long\def\cardinal#1{\left|#1\right|}
\global\long\def\calO{\mathcal{O}}
\global\long\def\mathordi#1{\mathord{#1}}
\global\long\def\Ezero{\EZero}
\global\long\def\xx{\mathbf{x}}

\global\long\def\NTPT{\operatorname{NTP}_{\operatorname{2}}}
\global\long\def\ist{\operatorname{ist}}
\global\long\def\C{\mathfrak{C}}
\global\long\def\alt{\operatorname{alt}}
\global\long\def\Ff{\mathbb{F}}
\global\long\def\Ll{\mathfrak{L}}
\global\long\def\calU{\mathcal{U}}
\global\long\def\Qq{\mathbb{Q}}
\global\long\def\AGLn{AGL_{n}\left(\Qq\right)}
\global\long\def\AGLt{AGL_{2}\left(\Qq\right)}
\global\long\def\AGL#1{AGL_{#1}\left(\Qq\right)}
\global\long\def\Zz{\mathbb{Z}}
\global\long\def\GL#1{GL_{#1}\left(\Qq\right)}
\global\long\def\bb{\mathfrak{b}}
\global\long\def\QqOmega{\Qq^{<\omega}}
\global\long\def\PGL#1{PGL_{#1}\left(\Qq\right)}
\global\long\def\PGaL#1#2{P\Gamma L_{#1}\left(#2\right)}
\global\long\def\AGalL#1#2{A\Gamma L_{#1}\left(#2\right)}

\title{Strict independence }

\author{Itay Kaplan and Alexander Usvyatsov }
\begin{abstract}
We investigate the notions of strict independence and strict non-forking,
and establish basic properties and connections between the two. In
particular it follows from our investigation that in resilient theories
strict non-forking is symmetric. Based on this study, we develop notions
of weight which characterize \NTP{}, dependence and strong dependence.
Many of our proofs rely on careful analysis of sequences that witness
dividing. We prove simple characterizations of such sequences in resilient
theories, as well as of Morley sequences which are witnesses. As a
by-product we obtain information on types co-dominated by generically
stable types in dependent theories. For example, we prove that every
Morley sequence in such a type is a witness.
\end{abstract}
\maketitle

\section{Introduction and preliminaries}

\subsection{Background}

This article is devoted to the study of two related notions: strict
independence and strict non-forking. We explore basic properties of
these concepts and establish connections between the two. 

Initially this work was carried out in the context of dependent theories.
Later we realized that many of the results carry over to a much more
general class of resilient theories \cite{cheBenya}. This is the
main setting of this paper. We will not, however, make any global
assumptions. Some of the results are relevant in a larger framework
of \NTP{} theories, and some are true for an arbitrary theory. We
state all the relevant assumptions locally. 

A central concept in our investigation, and a major technical tool,
is that of a \emph{witness}: an infinite sequence that always witnesses
dividing. We characterize witnesses in resilient theories in terms
of forking, and investigate their properties. 

\medskip{}

Strict independence is a new concept. This work is the first paper
where it is explicitly defined and systematically studied. However,
it has been used implicitly in several articles dealing to some extent
with different candidates for notions of weight in dependent theories,
such as \cite{Sh783,AlfAlex-dp,Us}. It was then isolated by the second
author in \cite{Us1} as a possible notion giving rise to a good concept
of orthogonality in dependent theories. 

Strict non-forking was defined by Shelah in \cite{Sh783}. The main
result that Shelah proves in regard to this notion is essentially
that in dependent theories strict non-forking implies strict independence
(in this paper, we refer to this result as ``Shelah's Lemma'').
Implicit in Shelah's paper is the following important fact: a strictly
non-forking sequence in a dependent theory is a witness (actually,
assuming \NTP{} and strict independence is enough). This fact was
later isolated and explicitly stated in \cite{kachernikov} and \textbf{\cite{Us1}}.\textcolor{magenta}{{}
}We refer to this fact here as ``Kim's Lemma for \NTP{} theories''.
The existence of such sequences was established in \cite{kachernikov}.

Existence of manageable witnesses is an important property of a theory.
For example, Kim showed that in a simple theory, all Morley sequences
are witnesses. This was a key technical result that allowed the development
of simple theories: e.g., it lies at heart of Kim's proof of symmetry
of non-forking. In \cite{kachernikov}, the existence of witnesses
was the main step in the proof of forking = dividing. So strict non-forking
and witnesses have already proved very useful in the study of \NTP{}
theories.

Nevertheless, very little was known in general about these concepts.
For example, it was not known whether or not strict non-forking is
symmetric. No natural characterization of the class of witnesses was
known either. The situation is remedied to a large extent by the current
investigation. 

If the theory $T$ is simple, then strict independence and strict
non-forking coincide, and both are equivalent to non-forking. It would
be very interesting to establish a strong connection between either
one of these concepts and non-forking in an arbitrary resilient or
at least dependent theory. We have attempted to address this question,
but have not yet succeeded to find a satisfactory answer.

\subsection{Overview}

Informally, we call a set of elements (tuples) \emph{strictly independent
}if every collection of indiscernible sequences starting with these
elements may be assumed mutually indiscernible. More precisely, a
set $\left\{ a_{i}\right\} $ is called strictly independent if for
any collection of indiscernible sequences $I_{i}$ where $I_{i}$
starts with $a_{i}$, there exist $I'_{i}$ of the same type as $I_{i}$
starting with $a_{i}$ such that $I'_{i}$ is indiscernible over $I'_{\neq i}$.
See Definition \ref{def:strictind}.

Strict non-forking is the forking notion (that is, the extensible
notion) that corresponds to the symmetrization of usual forking. That
is, $\tp\left(a/Ab\right)$ is a strictly non-forking extension of
$\tp\left(a/A\right)$ if for every $B$ containing $Ab$ there is
$a'\equiv_{Ab}a$ such that $a'\ind_{A}B$ and $B\ind_{A}^{d}a'$.
See Definition \ref{def:strictnonfork}. One can easily define strict
non-forking and strict Morley sequences as usual based on this notion,
Definition \ref{def:strictMorley}. 

As mentioned above, one central notion in our paper is that of a witness.
An indiscernible (or just 1-indiscernible)\textcolor{green}{{} }sequence
$I$ is called a witness for $a\in I$ if for any formula $\varphi(x,a)$,
whenever $\varphi(x,a)$ divides, $I$ witnesses it (see Definition
\ref{def:witness}). ``Kim's Lemma'' states that in a dependent
theory a strict Morley sequence is a witness. That is, one can think
of strict non-forking as a technical tool that gives rise to witnesses.
The first result in this paper is Lemma \ref{lem:main lemma}, a simple,
but an incredibly useful observation that in a sense ``reverses the
roles'' and allows to use witnesses in order to ``test'' for strict
non-forking. This result confirms the tight connection between the
concepts of strict non-forking and witnesses: one can use either one
in order to draw conclusions on the other.

``Shelah's Lemma'' says that in a dependent theory a strictly non-forking
sequence is also strictly independent. Since one of these notions
depends on the order, and the other does not, we can not expect the
converse to hold. Nevertheless, we prove (Theorem \ref{thm:Main})
that in a resilient theory tuples $a,b$ are strictly independent\textcolor{black}{{}
(a symmetric notion)} over an extension base $A$ if and only if $\tp\left(a/Ab\right)$
is a strictly non-forking extension of $\tp\left(a/A\right)$ if and
only if $\tp\left(b/Aa\right)$ is a strictly non-forking extension
of $\tp\left(b/A\right)$. In particular, we conclude that in resilient
theories strict non-forking is symmetric over extension bases. We
also strengthen ``Kim's Lemma'' by generalizing it to resilient
theories (Corollary \ref{cor:strict morley are witnesses in resilient}).

``Kim's Lemma'' suggests the following natural questions: Are there
convenient characterizations of sequences that are witnesses? Which
Morley sequences are witnesses? We answer these questions completely
under the assumption of resilience. For example, we show that \textcolor{black}{(over
an extension base)} a sequence $I=\left\langle a_{i}\right\rangle $
is a witness if and only if $\tp\left(a_{\neq i}/a_{i}\right)$ does
not fork; that is, $a_{\neq i}\ind a_{i}$, Theorem \ref{thm:characterize_witness}.
In addition, we show that a Morley sequence is strict Morley if and
only if it is a witness, Corollary \ref{cor:strictmorley_char}. This
confirms that ``Kim's Lemma'' is in a sense ``as good as it gets''. 

Before this work, it was not clear whether a strict Morley sequence
is necessarily ``totally strict'', that is, the sequences of pairs,
triplets, etc, are strict Morley sequences as well. Here we prove
a surprisingly simple and straightforward characterization of totally
strict Morley sequences as two-way Morley sequences (that is, the
sequence remains Morley when one reverses the order), Corollary \ref{cor:2way}.
We also present an example (due to Pierre Simon) of a strict Morley
sequence which is not totally strict (witnessed by the sequence of
pairs), Example \ref{exa:pierre}. This shows in particular that strict
non-forking can be sensitive to the ordering on the set: in Simon's
example one obtains a strict Morley sequence, which is not even a
Morley sequence once the order is reversed.

Under the further assumption that the theory is of bounded alternation
rank (e.g., dp-minimal), we present ``quantitative'' versions of
the results described above . For example, we calculate an explicit
bound on the length of a finite witness which is necessary to test
strict non-forking, Proposition \ref{prop:bounded version of main lemma}
(a quantitative version of Lemma \ref{lem:main lemma}). Similarly,
Corollary \ref{cor:dp-minimal witness} provides a bound on ``how
much non-forking'' one needs to produce a witness (a quantitative
version of Theorem \ref{thm:characterize_witness}). 

We proceed to drawing conclusions on types co-dominated by generically
stable types in dependent theories. In particular, we prove that any
Morley sequence in such a type is strict (hence is a witness), Theorem
\ref{thm:dombystable-strict}. As a corollary, one sees that non-forking
is symmetric on realizations of such types, Corollary \ref{cor:dombystable-symm}.

Section \ref{sec:My-weight-problem} is devoted to notions of weight
arising from the concepts of non-forking and independence studied
here. In particular, we characterize strong dependence (resp., dependence)
in terms of the appropriate weight being almost finite (resp., bounded). 

Finally, section \ref{sec:Problems} lists a few problems. 

Let us remark that even though simple theories are resilient, this
work gives little new information in this context. As shown in Theorem
\ref{thm:simple}, in simple theories forking and strict forking agree
(and as a matter of fact, this characterizes simplicity). Therefore
assuming simplicity, our results essentially amount to the following
(quite easy) observations: an indiscernible sequence is a witness
if and only if it is a Morley sequence, and a set is independent if
and only its elements start mutually indiscernible Morley sequences
if and only if it is strictly independent, Theorem \ref{thm:simple-allequivalent}.

We would like to thank Pierre Simon and Artem Chernikov for several
comments and suggestions. We would also like to thank the anonymous
referee for his useful remarks. In particular, we are thankful to
the referee for suggesting switching the emphasis from dependent theories
to a more general context. This suggestion led to a thorough revision
of the paper, during which we realized that more results can be generalized
to resilient theories than what we had originally thought.

\section{Preliminaries}

\subsection{Notations}

Our notations are standard: we are working with a complete first order
theory $T$, and $\C$ is its monster model.

By $a\ind_{A}^{f}B$ we mean that $\tp\left(a/BA\right)$ does not
fork over $A$, $a\ind_{A}^{d}B$ means that $\tp\left(a/BA\right)$
does not divide over $A$, $a\ind_{A}^{s}b$ means that $\tp\left(a/bA\right)$
does not strongly split over $A$ (see after Definition \ref{def:NTP})
and $a\ind_{A}^{i}b$ means that $\tp\left(a/bA\right)$ has a global
Lascar invariant extension over $A$. With the exception of Section
\ref{sec:My-weight-problem}, we shall omit the $f$ from $\ind^{f}$,
so $\ind$ will just mean non-forking.

If $I$ is a sequence (usually indiscernible), then $\varphi\left(x,I\right)$
denotes the set $\set{\varphi\left(x,a\right)}{a\in I}$.

\subsection{Preliminaries}

\subsubsection{Forking and dividing}
\begin{fact}
\label{fac:dividing} \cite[1.4]{SheSimple} The following are equivalent
for any theory $T$:
\begin{enumerate}
\item $a\ind_{A}^{d}b$.
\item For every indiscernible sequence $I$ over $A$ such that $b\in I$,
there is an indiscernible sequence $I'$ such that $I'\equiv_{Ab}I$
and $I'$ is indiscernible over $Aa$.
\item For every indiscernible sequence $I$ over $A$ such that $b\in I$,
there is $a'$ such that $a'\equiv_{Ab}a$ and $I$ is indiscernible
over $Aa'$.
\end{enumerate}
\end{fact}
\begin{claim}
\label{cla:preservationdividing} Assume that $a\ind_{B}b$ and $\varphi\left(x,b\right)$
divides / forks over $B$, then $\varphi\left(x,b\right)$ divides
/ forks over $Ba$ as well.\end{claim}
\begin{proof}
For dividing this follows from Fact \ref{fac:dividing}. Assume $\varphi\left(x,b\right)$
forks over $B$, so there are $n<\omega$, $\varphi_{i}\left(x,y_{i}\right)$
and $b_{i}$ for $i<n$ such that $\varphi_{i}\left(x,b_{i}\right)$
divides over $B$ and $\varphi\left(x,b\right)\vdash\bigvee_{i<n}\varphi_{i}\left(x,b_{i}\right)$.
We may assume $a\ind_{A}b\left\langle b_{i}\left|\, i<n\right.\right\rangle $.
So $\varphi_{i}\left(x,b_{i}\right)$ divides over $Ba$ and hence
$\varphi\left(x,b\right)$ forks over $Ba$. 
\end{proof}
We shall also use the fact that forking independence satisfies transitivity
(on the left):
\begin{fact}
(see e.g., \cite{Ad3})  if $a\ind_{A}bc$ then $a\ind_{Ab}c$. If
$b\ind_{A}c$ and $a\ind_{Ab}c$ then $ab\ind_{A}c$. \end{fact}
\begin{defn}
$ $
\begin{enumerate}
\item A set $A$ is called an \emph{extension base }if no type over $A$
forks over $A$. Equivalently, every type over $A$ has a global extension
which does not fork over $A$.
\item A theory $T$ is called \emph{extensible} if every set is an extension
base. 
\end{enumerate}
\end{defn}

\subsubsection{Forking in dependent and \NTP{} theories. }

Recall that a formula $\varphi\left(x,y\right)$ has the \emph{independence
property} if there are $\sequence{a_{i},b_{s}}{i<\omega,s\subseteq\omega}$
in $\C$ such that $\varphi\left(a_{i},b_{s}\right)$ iff $i\in s$.
A theory $T$ is \emph{dependent (NIP)} if no formula has the independence
property. 
\begin{defn}
\label{def:NTP}A theory has the\emph{ tree property of the second
kind} if there is a formula $\varphi\left(x,y\right)$, $k<\omega$
and an array $\left\langle a_{i,j}\left|\, i,j<\omega\right.\right\rangle $
such that:
\begin{itemize}
\item All rows are $k$-inconsistent: for all $i$, $\left\{ \varphi\left(x,a_{i,j}\right)\left|\, j<\omega\right.\right\} $
is $k$-inconsistent (every $k$-subset is inconsistent).
\item All vertical paths are consistent: for all $\eta:\omega\to\omega$,
$\left\{ \varphi\left(x,a_{i,\eta\left(i\right)}\right)\left|\, i<\omega\right.\right\} $
is consistent.
\end{itemize}
A theory is\emph{ \NTP{}} if it does not have the tree property of
the second kind. 
\end{defn}
Both simple and dependent theories are \NTP{}. 

We remind the reader a few basic properties of non-forking in dependent
and \NTP{} theories. For the definitions and more, the reader is
referred to \cite{kachernikov} and also to \cite{Ad,Sh:c}. 

Recall that a type $\tp\left(a/B\right)$ \emph{splits strongly} over
a set $A\subseteq B$ if there are tuples $b_{1},b_{2}\in B$ such
that for some formula $\varphi\left(x,y\right)$ over $A$ we have
$\varphi\left(x,b_{1}\right)\land\neg\varphi\left(x,b_{2}\right)\in\tp\left(a/B\right)$
and there is an $A$-indiscernible sequence containing both $b_{1}$
and $b_{2}$. 
\begin{fact}
\label{fac:strongsplit}\cite{Sh783} ($T$ dependent) If $\tp(a/B)$
splits strongly over a set $A\subseteq B$, then $\tp(a/B)$ divides
over $A$. More generally, if $c,d\in B$ have the same Lascar strong
type over $A$, and $ac\not\equiv_{A}ad$ then $\tp\left(a/B\right)$
forks over $A$. In particular this is true when $A$ is a model and
$c\equiv_{A}d$. This means that if $p$ is a global type that does
not fork over a model $M$, then it is invariant over $M$. 
\end{fact}

\begin{fact}
(Preservation of indiscernibility)  Let $I$ be an indiscernible sequence
over a set $A$, and assume $b\ind_{A}^{i}I$. Then $I$ is indiscernible
over $Ab$. By Fact \ref{fac:strongsplit}, when $T$ is dependent,
$\ind=\ind^{i}$, hence the same holds with $\ind$.
\end{fact}
We recall some basic facts on invariant types:
\begin{defn}
Suppose $p$ is a global type which is invariant over a set $A$. 
\begin{enumerate}
\item We say that a sequence $\left\langle a_{i}\left|\, i<\alpha\right.\right\rangle $
is generated by $p$ over $B\supseteq A$ if $a_{0}\models p|_{B}$
and for all $i<\alpha$, $a_{i}\models p|_{Ba_{<i}}$. Note that this
sequence is indiscernible over $B$.
\item We let the type $p^{\left(\alpha\right)}$ be the union of $\tp\left(\left\langle a_{i}\left|\, i<\alpha\right.\right\rangle /B\right)$
running over all $B\supseteq A$, where $\left\langle a_{i}\left|\, i<\alpha\right.\right\rangle $
is some (any) sequence generated by $p$ over $B$. Note that this
type is well defined and invariant over $A$.
\end{enumerate}
\end{defn}
Note that models are always extension bases, but there are dependent
theories which are not extensible. Most of our results in this paper
assume that the sets of parameters one works over are extension bases.
The reason this assumption is helpful is the following result (due
to Chernikov and the first author) which will be explicitly and implicitly
used throughout the article:
\begin{fact}
\cite{kachernikov} ($T$ \NTP{}) Let $A$ be an extension base.
Then dividing over $A$ equals forking over $A$. More precisely,
if a formula $\varphi\left(x,a\right)$ forks over $A$, then it divides
over $A$. 
\end{fact}
Note that if $A$ is not an extension base, then the conclusion of
the fact above can not possibly hold (since no type over $A$ divides
over $A$); so the theorem tells us that in \NTP{} theories this
is the only obstacle. 

We will refer to this fact as ``forking = dividing over $A$'' (or
omit $A$ when it is clear from the context).\textcolor{green}{{} }
\begin{cor}
\label{cor:Left Extension}\cite{kachernikov} (T \NTP{}) Let $A$
be an extension base. Then forking has left extension over $A$, which
means: if $a\ind_{A}b$ then for all $c$ there is $c'\equiv_{Aa}c$
such that $ac'\ind_{A}b$.
\end{cor}
Finally, we define the class of resilient theories which was introduced
in \cite{cheBenya}. 
\begin{defn}
A theory $T$ is called \emph{resilient }if for any indiscernible
sequence, $I=\left\langle a_{i}\left|\, i\in\mathbb{Z}\right.\right\rangle $
and formula $\varphi\left(x,y\right)$, if $\varphi\left(x,a_{0}\right)$
divides over $a_{\neq0}$, then $\varphi\left(x,I\right)$ is inconsistent. \end{defn}
\begin{fact}
\cite{cheBenya} Simple and dependent theories are resilient. Resilient
theories are \NTP{}. 
\end{fact}

\section{Basic notions}

We now move on to the definitions of the main notions studied in this
paper. The following definitions appear in \cite{Sh783}:
\begin{defn}
\label{def:strictnonfork}We say that $\tp\left(a/Ab\right)$ \emph{strictly
does not fork }over $A$ (we write $a\ind_{A}^{\st}b$) if there is
a global extension $p$ of $\tp\left(a/Ab\right)$ which does not
fork over $A$ (so in particular $a\ind_{A}b$), and for any $B\supseteq Ab$,
if $c\models p|_{B}$ then $\tp\left(B/Ac\right)$ does not divide%
\footnote{It is a matter of taste whether it should be ``divide'' or ``fork''.
For all the results here there is no difference except Lemma \ref{lem:stind-eqdef}
where minor changes are required. %
} over $A$ (i.e., $B\ind_{A}^{d}c$). \end{defn}
\begin{rem}
If $a\ind_{A}^{\st}b$ and $c,d\in A$ then $ac\ind_{A}^{\st}bd$
($c,d$ may be empty).\end{rem}
\begin{lem}
\label{lem:stind-eqdef}  The following are equivalent for a tuple
$a$ and sets $A\subseteq B$:
\begin{enumerate}
\item $a\ind_{A}^{\st}B$ 
\item For every $c$, there is some $c'\equiv_{B}c$ such that $Bc'\ind_{A}^{d}a$
and $a\ind_{A}Bc'$.
\item For every $\left(\left|B\right|+\left|T\right|\right)^{+}$-saturated
model $M$ containing $B$, there is some $M'\equiv_{B}M$ such that
$M'\ind_{A}^{d}a$ and $a\ind_{A}M'$.
\item There is some $\left(\left|B\right|+\left|T\right|\right)^{+}$-saturated
model $M$ containing $B$, such that $M\ind_{A}^{d}a$ and $a\ind_{A}M$.
\item For every finite sub-tuple $a_{0}$ of $a$ and finite subset $B_{0}$
of $B$, $a_{0}\ind_{A}^{\st}B_{0}$. 
\item There is some $\left(\left|A\right|+\left|T\right|\right)^{+}$-saturated
model $M$ containing $B$, such that $M\ind_{A}^{d}a$ and $a\ind_{A}M$.
\end{enumerate}
\end{lem}
\begin{proof}
Clearly (1) $\Rightarrow$ (2) (by applying an automorphism over $A$)
and (2) $\Rightarrow$ (3) $\Rightarrow$ (4).

(4) $\Rightarrow$ (1): We must show that there is a global type $p\left(x\right)$
extending $r\left(x\right):=\tp\left(a/B\right)$ such that $p\left(x\right)$
does not fork over $A$ and $\neg\psi\left(d,x\right)\in p\left(x\right)$
whenever $\psi\left(y,a\right)$ divides over $A$. In other words,
the following set is consistent: 
\begin{eqnarray*}
r\left(x\right) & \cup & \left\{ \neg\varphi\left(x,c\right)\left|\,\varphi\in L\left(A\right),\,\varphi\left(x,c\right)\mbox{ forks over }A\right.\right\} \\
 & \cup & \left\{ \neg\psi\left(d,x\right)\left|\,\psi\in L\left(A\right),\,\psi\left(y,a\right)\mbox{ divides over }A\right.\right\} .
\end{eqnarray*}
If not, $r\left(x\right)\models\bigvee\varphi_{i}\left(x,c_{i}\right)\vee\bigvee\psi_{j}\left(d_{j},x\right)$
for such formulas. But then by saturation of $M$ we may assume $c_{i},d_{j}\in M$
contradicting (4). 

Easily (1) implies (5). The converse follows from the same argument
as above, i.e., we only need to show that the set above is consistent,
and by (5) every finite subset is. 

Finally, (6) follows from (3), and implies (5) by (4). \end{proof}
\begin{cor}
If $M$ is $\left|A\right|^{+}$-saturated, then for any $a$, $a\ind_{A}^{\st}M$
iff {[}$M\ind_{A}^{d}a$ and $a\ind_{A}M${]}. Hence if forking =
dividing over $A$, $\ind_{A}^{\st}$ is symmetric for $\left|A\right|^{+}$-saturated
models. \end{cor}
\begin{proof}
Clearly left implies right. Conversely, see clause (6) of Lemma \ref{lem:stind-eqdef}. 
\end{proof}
Recall that a sequence $I=\langle a_{i}\rangle$ is called \emph{non-forking}
over a set $A$ if $a_{i}\ind_{A}a_{<i}$ for all $i$. It is called
\emph{Morley }over $A$ if it is both indiscernible over $A$ and
non-forking over $A$. By analogy, we define:
\begin{defn}
\label{def:strictMorley}A sequence $I=\langle a_{i}\rangle$ is called
\emph{strictly non-forking} over a set $A$ if $a_{i}\ind_{A}^{\st}a_{<i}$
for all $i$. It is called \emph{a strict Morley sequence }over $A$
if it is both indiscernible over $A$ and strictly non-forking over
$A$. It is called \emph{totally strict }if in addition for all $n<\omega$,
the sequence of increasing $n$-tuples from it is also strict Morley
(i.e., if the order type of $I$ is $\omega$, then 
\[
\left\langle \left(a_{j}\left|\, ni<j<n\left(i+1\right)-1\right.\right)\left|\, i<\omega\right.\right\rangle 
\]
 is strict Morley). 
\end{defn}
Of course, if $I$ is a strict Morley sequence over $A$, then it
is Morley over $A$.
\begin{example}
\label{exa:rationals}We will see later (after Corollary \ref{cor:2way})
that for $T=Th\left(\mathbb{Q},<\right)$, every non-constant $A$-indiscernible
sequence of singletons is totally strict over $A$. 
\end{example}
The notion of strict non-forking fits well with dependent theories.
The analog for \NTP{} is:
\begin{defn}
We say that $\tp\left(a/Ab\right)$ is \emph{strictly invariant} over
$A$ (we write $a\ind_{A}^{\ist}b$) if there is a global extension
$p$ of $\tp\left(a/Ab\right)$ which is Lascar invariant over $A$
(i.e., does not strongly split over $A$), and for any $B\supseteq Ab$,
if $c\models p|_{B}$ then $\tp\left(B/Ac\right)$ does not divide
over $A$.
\end{defn}
Observe that in general, if $a\ind_{A}^{\ist}b$ then $a\ind_{A}^{\st}b$
and the converse is true in the context of NIP in light of Fact \ref{fac:strongsplit}. 

We can define a strictly invariant sequence, strictly invariant Morley
sequences and totally strict invariant sequences using $\ind^{\ist}$
instead of $\ind^{\st}$, and an analog to Lemma \ref{lem:stind-eqdef}
holds. 

Finally, we give the main new definition of this paper:
\begin{defn}
\label{def:strictind}Let $B$ be a set, and $\mathcal{J}$ a set
of tuples from $\C$. We say that $\mathcal{\mathcal{J}}$ is \emph{strictly
independent }over $B$ if the following holds: for every sequence
$\mathcal{I}=\sequence{I_{a}}{a\in\mathcal{J}}$ of infinite $B$-indiscernible
sequences such that for each $a\in\mathcal{J}$, $a\in I_{a}$, there
is $\mathcal{I}'=\sequence{I'_{a}}{a\in\mathcal{J}}$ satisfying:
\begin{itemize}
\item $I_{a}\equiv_{Ba}I'_{a}$ for all $a\in\mathcal{J}$.
\item $I'_{a}$ is indiscernible over $BI'_{\neq a}$.
\end{itemize}
\end{defn}
We will refer to the second item in the definition above as ``$I'_{a}$
are mutually indiscernible over $B$'', or ``$I'_{a}$ are mutually
$B$-indiscernible''.
\begin{rem}
It is clear from the definition that if $\left\{ a_{i}\left|\, i<\lambda\right.\right\} $
is strictly independent over $B$, and $b_{i}\in B$ for $i<\lambda$,
then $\left\{ a_{i}b_{i}\left|\, i<\lambda\right.\right\} $ is also
strictly independent. It is also easy to see that it is enough to
consider $\omega$-sequences in $\calI$ and assume that each $I_{a}$
starts with $a$. \end{rem}
\begin{claim}
\label{cla:making mutually indiscernible}Suppose that $\sequence{I_{i}}{i<\lambda}$
is a sequence of indiscernible sequences over a set $B$ such that
$I_{i}=\sequence{a_{i,j}}{j<\omega}$, $p=\tp\left(\sequence{a_{i,0}}{i<\lambda}/B\right)$,
and: 
\begin{itemize}
\item [(*)]For all $\eta:\lambda\to\omega$, $\sequence{a_{i,\eta\left(i\right)}}{i<\omega}\models p$. 
\end{itemize}
Then there are mutually indiscernible sequences $\sequence{I_{i}'}{i<\omega}$
such that $I_{i}'\equiv_{a_{i,0}B}I_{i}$.\end{claim}
\begin{proof}
By Ramsey and compactness (see e.g., \cite[Lemma 5.1.3]{TentZiegler})
we can extract mutually indiscernible sequences $\bar{I}''=\left\langle I_{i}''\left|\, i<\lambda\right.\right\rangle $
such that $\sequence{a_{i,0}}{i<\omega}\equiv_{B}\sequence{a_{i,0}''}{i<\omega}$
(where $I_{i}''$ starts with $a_{i,0}''$) and such that $(*)$ above
holds for $\bar{I}''$. Now apply an automorphism taking the first
column to $\left\langle a_{i,0}\left|\, i<\lambda\right.\right\rangle $. 
\end{proof}
One of the motivations for defining strict independence, and for investigating
its connection to strict non-forking is the following lemma due to
Shelah:
\begin{lem}
\label{lem:(Shelah's-Lemma)}(Shelah's Lemma, \cite[Claim 5.13]{Sh783})
Let $\mathcal{J}=\left\langle a_{i}\left|\, i<\lambda\right.\right\rangle $
be a strictly invariant sequence over a set $A$. Then $\mathcal{J}$
is strictly independent. \end{lem}
\begin{proof}
\renewcommand{\qedsymbol}{} 

For the proof define:
\begin{defn*}
Let $A$ be a set, and $\mathcal{J}=\left\langle a_{i}\left|\, i<\lambda\right.\right\rangle $
a sequence of tuples. We say that $\mathcal{J}$ is \emph{almost}
\emph{strictly independent }over $A$ if the following holds: for
every set $\mathcal{I}=\left\{ I_{i}\left|\, i<\lambda\right.\right\} $
of $A$-indiscernible sequences of length $\omega$ such that $I_{i}$
starts with $a_{i}$ , there is $\mathcal{I}'=\left\{ I'_{i}\left|\, i<\lambda\right.\right\} $
satisfying:
\begin{itemize}
\item $I_{i}\equiv_{Aa_{i}}I'_{i}$ for all $i<\lambda$.
\item $I'_{i}$ is indiscernible over $AI'_{<i}a_{>i}$.
\end{itemize}
\end{defn*}
For simplicity of notation, assume without loss of generality that
$A=\emptyset$. 

By Claim \ref{cla:making mutually indiscernible} (alternatively,
see Lemma 2.4 in \cite{OnUs2}), it follows that if $\mathcal{J}$
is almost strictly independent (as a sequence) then it is strictly
independent (as a set). Hence it is enough to prove:
\begin{claim*}
$\mathcal{J}$ is almost strictly independent. \end{claim*}
\begin{proof}
\renewcommand{\qedsymbol}{$\square$} Suppose $\mathcal{I}=\left\{ I_{i}\left|\, i<\lambda\right.\right\} $
is a set of indiscernible sequences such that $I_{i}$ starts with
$a_{i}$. By compactness we may assume $\lambda=n<\omega$. The proof
is by induction on $n$. Suppose we have $I_{i}'$ for $i<n$ and
consider $n+1$. Perhaps changing $I_{i}'$, we may assume that $a_{n}\ind^{\ist}I_{<n}'$.
Hence also $I_{<n}'\ind^{d}a_{n}$. By Fact \ref{fac:dividing}, there
is an indiscernible sequence $I'_{n}$ such that $I_{n}\equiv_{a_{n}}I_{n}'$
and $I_{n}'$ is indiscernible over $I_{<n}'$. Now, $I_{n}'$ is
indiscernible over $I_{<n}'$ by construction. For every $i<n$, $I_{i}'$
is indiscernible over $I_{<i}'a_{>i}$ by the induction hypothesis
(where $a_{>i}=a_{i+1}\ldots a_{n-1}$). As $a_{n}\ind^{i}I'_{<n}$,
, it follows that $a_{n}\ind_{I_{<i}'a_{>i}}^{i}I_{i}'$, so by preservation
of indiscernibility, $I_{i}'$ is indiscernible over $I_{<i}'a_{>i}a_{n}$. 
\end{proof}
\end{proof}
Strict non-forking, and especially strict Morley sequences have already
proved very useful due to the following observation, shown independently
in \cite{Us1} and \cite{kachernikov}%
\footnote{In \cite{kachernikov} it is stated for strict invariant sequences,
but the proof there really uses Shelah's lemma and Lemma \ref{lem:(Kim's-Lemma)}
as stated here.%
} (in fact it is implicit in \cite[Claim 5.8]{Sh783}): 
\begin{lem}
\label{lem:(Kim's-Lemma)}(``Kim's Lemma'') ($T$ \NTP{}) If $\left\langle a_{i}\left|\, i<\left|T\right|^{+}\right.\right\rangle $
is a strictly independent sequence over $B$, and $\varphi_{i}\left(x,a_{i}\right)$
divides over $B$, then $\left\{ \varphi_{i}\left(x,a_{i}\right)\left|\, i<\left|T\right|^{+}\right.\right\} $
is inconsistent. In particular, this is true if $\left\langle a_{i}\left|\, i<\left|T\right|^{+}\right.\right\rangle $
is strictly invariant (or strictly non-forking in case of NIP). 
\end{lem}
It is therefore interesting to understand strict Morley sequences
and strictly non-forking extensions in general. From the definitions
it is not even clear that such extensions exist. However, a few existence
results have been shown:
\begin{fact}
\label{fac:strict_exist}$ $
\begin{itemize}
\item \cite{kachernikov}($T$ \NTP{}) Let $A$ be an extension base for
$\ind^{i}$. Then $A$ is a strict invariance extension base. That
is, every type over $A$ has a global strictly invariant extension.
\item ($T$ \NTP{}) Let $A$ be an extension base. Then $A$ is a strict
extension base. That is, every type over $A$ has a global strictly
non-forking extension. {[}This is not mentioned explicitly but follows
from \cite[Proposition 3.7]{kachernikov}.{]} 
\item \cite{Us1} Let $p$ be a global type which is both an heir and non-forking
over $A$. Then $p$ is a global strictly invariant extension of $p|_{A}$.
\end{itemize}
\end{fact}
\begin{rem}
\label{remark strict-morley-sequences-exist}Note that if $M\supseteq A$
is a model, and $p$ is a global type that is Lascar invariant $A$,
then $p$ is invariant over $M$. We can conclude that under \NTP{}
for every $\ind^{i}$-extension base $A$ (under NIP, every extension
base) and a tuple $a$, there is a strictly invariant Morley sequence
over $A$ starting with $a$ (if $p$ is a global strictly invariant
type containing $\tp\left(a/A\right)$, then find some model $M\supseteq A$
and a conjugate (over $A)$ $q$ of $p$ such that $a\models q|_{M}$
and define $a_{i}\models q|_{Ma_{<i}}$ to get a strict invariant
Morley sequence). 

Similarly, under \NTP{}, if $A$ is an extension base, then for any
tuple $a$ there is a strict Morley sequence starting with $a$ over
$A$: take a global strict non-forking type $p\left(x\right)$ extending
$\tp\left(a/A\right)$ and generate a long enough sequence using $p$,
then use the ``Morley-Erd\H{o}s-Rado'' theorem (see e.g., \cite[Proposition 1.6]{casanovas-simple})
in order to ``extract'' an indiscernible sequence. 
\end{rem}

\section{\label{sec:Basic-properties}The main lemma}
\begin{defn}
\label{def:witness}Call an infinite sequence $I$ a \emph{witness
for $a$ over $A$}, if:
\begin{enumerate}
\item For all $b\in I$ we have $a\equiv_{A}b$.
\item For any formula $\varphi\left(x,y\right)$ over $A$, if $\varphi\left(x,a\right)$
divides over $A$, then $\varphi(x,I_{0})$ is inconsistent for every
countable $I_{0}\subseteq I$. 
\end{enumerate}
Say that $I$ is an \emph{indiscernible witness} \emph{for} $a$ \emph{over
$A$} if $I$ is a witness for $a$ over $A$, and it is indiscernible
over $A$. \end{defn}
\begin{rem}
Note that if $I$ is a witness for some $a$ over $A$, then it is
a witness for every $a\in I$, and, in fact, for any $a$ realizing
the type of some (any) element of $I$ over $A$. So we will often
simply say ``$I$ is a witness over $A$'', omitting $a$. \end{rem}
\begin{example}
\label{exa:strict Morley is a witness}($T$ \NTP{}) By Lemma \ref{lem:(Kim's-Lemma)},
strictly independent sequences of tuples of the same type over $A$
are witnesses over $A$. \end{example}
\begin{obs}
\label{obs:direct consequence of Kim} ($T$ \NTP{}) Let $A$ be
an extension base. If there exists a witness $I=\left\langle b_{i}\left|\, i<\omega\right.\right\rangle $
for $b=b_{0}$ over $A$, which is indiscernible over $Aa$, then
$a\ind_{A}b$.\end{obs}
\begin{proof}
If $a\nind_{A}b$, then (since forking = dividing over $A$) there
is some formula $\varphi\left(x,y\right)$ over $A$ such that $\varphi\left(a,b\right)$
holds and $\varphi\left(x,b\right)$ divides over $A$. So, by definition,
$I$ witnesses it. But this is a contradiction --- by indiscernibility
over $Aa$, $\varphi\left(a,b_{i}\right)$ holds for all $i<\alpha$.
\end{proof}
The following lemma is quite easy, but incredibly useful. In a sense
it is an analogue of Observation \ref{obs:direct consequence of Kim}
for $\ind^{\st}$ and $\ind^{\ist}$. 
\begin{lem}
\label{lem:main lemma}(The main Lemma) Assume that $T$ is \NTP{}
and let $A$ be an extension base.

Suppose $I=\left\langle a_{i}\left|\, i<\omega\right.\right\rangle $
is a witness for $a=a_{0}$ over $A$. Then, if $I\ind_{A}b$ and
$I$ is indiscernible over $Ab$, then $a\ind_{A}^{\st}b$. 

Suppose $I=\left\langle a_{i}\left|\, i<\omega\right.\right\rangle $
is a witness for $a=a_{0}$ over $A$. Then, if $I\ind_{A}^{i}b$
and $I$ is indiscernible over $Ab$, then $a\ind_{A}^{\ist}b$. \end{lem}
\begin{proof}
(1) It is enough to prove (recalling Lemma \ref{lem:stind-eqdef})
that for any $c$, there is some $c'\equiv_{Ab}c$ such that $bc'\ind_{A}a$
and $a\ind_{A}bc'$. Let $p\left(x\right)=\tp\left(c/Ab\right)$.
So it is enough to show that the following set is consistent: 
\begin{eqnarray*}
p\left(x\right) & \cup & \left\{ \neg\varphi\left(x,b,a\right)\left|\,\varphi\in L\left(A\right),\,\varphi\left(x,y,a\right)\mbox{ forks over }A\right.\right\} \\
 & \cup & \left\{ \neg\psi\left(a,b,x\right)\left|\,\psi\in L\left(A\right),\,\psi\left(z,b,c\right)\mbox{ forks over }A\right.\right\} 
\end{eqnarray*}

Suppose not, then $p\left(x\right)\vdash\varphi\left(x,b,a\right)\vee\psi\left(a,b,x\right)$
where $\varphi\left(x,y,a\right)$ forks over $A$, and $\psi\left(z,b,c\right)$
forks over $A$ (recall that forking formulas form an ideal, that
is, forking is preserved under finite disjunctions). Since forking
= dividing we have that $\varphi\left(x,y,a\right)$ and $\psi\left(z,b,c\right)$
actually divide over $A$. 

Now, as $I$ is a witness for $a$, it witnesses that $\varphi\left(x,y,a\right)$
divides over $A$. Clearly, $\left\{ \varphi\left(x,b,a_{i}\right)\left|\, i<\omega\right.\right\} $
is inconsistent. As $I$ is indiscernible over $Ab$, it follows that
$p\left(x\right)\vdash\bigvee_{i<n}\psi\left(a_{i},b,x\right)$ for
some $n<\omega$. Take some $c'\models p\left(x\right)$ such that
$I\ind_{A}bc'$. So for some $i$, the formula $\psi\left(a_{i},b,c'\right)$
holds, so $\psi\left(z,b,c'\right)$ does not divide over $A$, therefore
neither does $\psi\left(z,b,c\right)$ --- a contradiction.

(2) Similar, but now instead of one formula $\psi$ we will have finitely
many formulas of the form $\psi\left(a,b,x_{1}\right)\land\neg\psi\left(a,b,x_{2}\right)$
where $x_{1}$ and $x_{2}$ are sub-tuples of $x$ and both $c\upharpoonright x_{1}$
and $c\upharpoonright x_{2}$ have the same Lascar strong type over
$A$. \end{proof}
\begin{rem}
\label{rem:weakining}Both Observation \ref{obs:direct consequence of Kim}
and Lemma \ref{lem:main lemma} hold even when we weaken the assumption
of indiscernibility to ``having the same type''. So, if $A$ is
an extension base, $I=\left\langle a_{i}\left|\, i<\omega\right.\right\rangle $
is a witness for $a=a_{0}$ over $A$, $I\ind_{A}b$, and all tuples
in $I$ have the same type over $Ab$, then $a\ind_{A}^{\st}b$. The
analog holds for $\ind^{\ist}$. 
\end{rem}

\section{On witnesses and strict Morley sequences.}

We begin with a characterization of indiscernible sequences which
are witnesses over a given extension base $C$.%

\begin{prop}
\label{prop:witness alternation rank} ($T$ resilient) Let $A$ be
any set. If $I=\left\langle a_{i}\left|\, i<\omega\right.\right\rangle $
is an indiscernible sequence over $A$ such that $a_{\neq i}\ind_{A}a_{i}$
for all $i<\omega$, and $\varphi\left(x,a_{0}\right)$ divides over
$A$ then $\left\{ \varphi\left(x,a_{i}\right)\left|\, i<\omega\right.\right\} $
is inconsistent.
\end{prop}

\begin{proof}
First we extend $I$ to an $A$-indiscernible sequence $I'=\sequence{a_{i}}{i\in\Zz}$.
By indiscernibility, it is still true that $a_{\neq i}\ind_{A}a_{i}$
for all $i\in\Zz$. If $\varphi\left(x,a_{0}\right)$ divides over
$A$, then it also divides over $Aa_{\neq i}$ by Claim \ref{cla:preservationdividing}.
By definition of resilience, this means that $\varphi\left(x,I\right)$
is inconsistent.\end{proof}
\begin{thm}
\label{thm:characterize_witness}(Characterization of indiscernible
witnesses) Suppose $T$ is resilient. Let $C$ be an extension base,
$I=\sequence{a_{i}}{i\in X}$ an infinite indiscernible sequence over
$C$. Then $I$ is a witness over $C$ if and only if for every $i\in X$
we have $a_{\neq i}\ind_{C}a_{i}$. \end{thm}
\begin{proof}
First, assume $I$ is a witness, but for some $i\in X$ we have $a_{\neq i}\nind_{C}a_{i}$.
We may assume that $X=\left(\Qq,<\right)$. Since forking = dividing
over $C$, there is a formula $\varphi\left(x,y\right)\in L\left(C\right)$
such that $\varphi\left(x,a_{i}\right)$ divides over $C$ and for
some $a=a_{j_{1}}\ldots a_{j_{k}}$ from $I$, $\C\models\varphi\left(a,a_{i}\right)$.
Enumerate $a$ such that $j_{1}<j_{2}<\ldots<j_{\ell}<i<j_{\ell+1}<\ldots<j_{k}$
(it is possible that $j_{\ell}=-\infty$ or $j_{\ell+1}=+\infty$).
Then since $I$ is a witness, by indiscernibility we have that $\left\langle a_{q}\left|\, q\in\left(j_{\ell},j_{\ell+1}\right)\right.\right\rangle $
witnesses the dividing of $\varphi\left(x,a_{i}\right)$ ; however,
by indiscernibility again, $\C\models\varphi\left(a,a_{q}\right)$
for all $q\in\left(j_{\ell,}j_{\ell+1}\right)$, a contradiction.

The converse follows from Proposition \ref{prop:witness alternation rank}. 
\end{proof}
Note that the ``only if'' direction holds in an arbitrary theory,
assuming only that dividing = forking over $C$. So for example it
would hold assuming that $T$ is \NTP{}.
\begin{cor}
\label{cor:strict morley are witnesses in resilient} ($T$ resilient)
Strict non-forking sequences are witnesses over extension bases. \end{cor}
\begin{proof}
Suppose that $I=\left\langle a_{i}\left|\, i<\omega\right.\right\rangle $
is a strict Morley sequence over an extension base $A$. Then for
any $i$, $a_{i}\ind_{C}^{\st}a_{<i}$, so $a_{<i}\ind_{C}a_{i}$.
In addition, $a_{>i}\ind_{C}a_{\leq i}$ since $I$ is Morley (by
transitivity). By transitivity, $a_{\neq i}\ind a_{i}$.\end{proof}
\begin{rem}
\label{rem:strict morley are witnesses} Note that Corollary \ref{cor:strict morley are witnesses in resilient}
is \uline{not} a special case of Example \ref{exa:strict Morley is a witness}. 
\end{rem}
Recalling the second part of Remark \ref{remark strict-morley-sequences-exist}
we also get the following corollary:
\begin{cor}
\label{cor:witnesses exist}(T resilient) Let $A$ is an extension
base, then for every $a$ there is some witness starting with $a$
over $A$. 
\end{cor}
Moreover, we obtain a sort of converse to Lemma \ref{lem:main lemma}:
\begin{prop}
\label{prop:converse to main lemma}($T$ resilient) Suppose that
$A$ is an extension base and $a\ind_{A}^{\st}b$. Then there is a
witness $I=\sequence{a_{i}}{i<\omega}$ over $A$ starting with $a_{0}=a$
which is indiscernible over $Ab$ and such that $I\ind_{A}b$. \end{prop}
\begin{proof}
Let $p\left(x\right)$ be a global strictly non-forking type extending
$\tp\left(a/bA\right)$. Using $p$ generate a long sequence $I'=\sequence{a_{i}'}{i<\kappa}$,
it follows that $I'$ is strictly non-forking over $A$ and by transitivity,
$I'\ind_{A}b$. Now use the Erd\H{o}s-Rado theorem as in Remark \ref{remark strict-morley-sequences-exist}
to find a strict Morley sequence $I$, indiscernible over $Ab$ such
that $I\ind_{A}b$. By Theorem \ref{cor:strict morley are witnesses in resilient}
$I$ is a witness. Together we are done. 
\end{proof}
From Lemma \ref{lem:main lemma} we conclude the following characterization
of strict Morley sequences among Morley sequences:
\begin{cor}
\textbf{\label{cor:strictmorley_char} }($T$ resilient) Let $C$
be an extension base. A Morley sequence $I=\left\langle a_{i}\left|\, i<\omega\right.\right\rangle $
is a strict Morley sequence over $C$ iff it is a witness over $C$. \end{cor}
\begin{proof}
We already know that strict Morley sequences are witnesses. Conversely,
assuming $I$ is Morley, by transitivity, $a_{\geq i}\ind_{C}a_{<i}$.
This means that $\left\langle a_{j}\left|\, j\geq i\right.\right\rangle $
is a witness which is indiscernible over, and independent from $\left\langle a_{j}\left|\, j<i\right.\right\rangle $
over $C$. By Lemma \ref{lem:main lemma} we are done. 
\end{proof}
As a corollary of the proof, we obtain a surprising fact: a two-way
Morley sequence is strict Morley (even totally strict):
\begin{cor}
\label{cor:2way}($T$ resilient) For an extension base $A$ and an
$A$-indiscernible sequence $I=\left\langle a_{i}\left|\, i<\omega\right.\right\rangle $
the following are equivalent:
\begin{enumerate}
\item $I$ is a two-way independent sequence over $A$ (i.e., $a_{\geq i}\ind_{A}a_{<i}$
and $a_{<i}\ind_{A}a_{\geq i}$).
\item $I$ is totally strict.
\end{enumerate}
\end{cor}
Going back to Example \ref{exa:rationals}, it is easy to see that
every $A$-indiscernible increasing sequence of singletons $I=\left\langle a_{i}\left|\, i<\omega\right.\right\rangle $
is Morley (because some increasing Morley sequence over $A$ exists
in the type of $a_{0}$ over $A$: the one generated by the invariant
global type defined by $x>c$ iff $\exists c'\left(c'\geq c\land c'\equiv_{A}a_{0}\right)$
--- and all such sequences have the same type over $A$). Analogously,
this is true for decreasing sequences as well, so $I$ is two-way
independent over $A$.

Here we present an example (due to Pierre Simon) of a strict Morley
sequence which is not totally strict. The condition above is, therefore,
not a characterization of strict Morley sequences.
\begin{example}
\label{exa:pierre}Let $T$ be the theory of dense trees in the language
$<,\wedge$ where $\wedge$ is the meet function. This theory is the
model completion of the theory of trees. It is $\omega$-categorical
and has elimination of quantifiers. 

Note that there are only five types of a non-constant indiscernible
sequence $\sequence{a_{i}}{i<\omega}$ in this theory: increasing,
decreasing ($a_{n}\frac{>}{<}a_{n+1}$), increasing / decreasing comb
($a_{n+2}\wedge a_{n+1}\frac{>}{<}a_{n+1}\wedge a_{n}$) and flower
($a_{n+1}\wedge a_{n}=a_{n+1}\wedge a_{n+2}$). 

Let $M$ be a countable model containing the tree $2^{<\omega}$.
One of the branches, say $\eta$ in $2^{\omega}$ does not have a
point above it in $M$. Let $p\left(x\right)$ be the type $\left\{ x>\eta\upharpoonright i\left|\, i<\omega\right.\right\} $.
Note that this already determines a complete type over $M$. Let $q$
be a global strictly non-forking extension, then $q^{\left(2\right)}$
is not strictly non-forking: 

Assume not, and let $\sequence{a_{i}}{i<\omega}\models q^{\left(\omega\right)}|_{M}$.
Then it is an indiscernible sequence. Note that the formulas $x>a_{0}$
and $x>a_{0}\wedge a_{1}$ divide over $M$. By analyzing the possible
types of $\sequence{a_{i}}{i<\omega}$ it follows that the sequence
of pairs cannot be strictly non-forking. 
\end{example}

\section{Symmetry of strict non-forking}

As a corollary of the previous sections we answer the main question
that motivated this paper: in resilient theories we have equivalence
between strict non-forking and strict independence for a 2-element
set (the original question that we wanted to address was on dependent
theories, but while during the last revision we managed to prove this
for resilient as well). It follows in particular that strict non-forking
is \emph{symmetric}.\emph{ }

First we introduce the following notion, which is (in general) weaker
than strict (or almost strict) independence. 
\begin{defn}
\label{def:wintenss independent}Let $B$ be a set, and $\mathcal{J}$
a set of tuples from $\C$. We say that $\mathcal{\mathcal{J}}$ is
\emph{witness independent }over $B$ if there exists sequences $\sequence{I_{a}}{a\in\mathcal{J}}$
of length $\omega$ such that: 
\begin{itemize}
\item $I_{a}$ starts with $a$.
\item Each $I_{a}$ is a witness for $a$ over $B$.
\item The sequences $I_{a}$ for $a\in\mathcal{J}$ are mutually indiscernible
over $B$. 
\end{itemize}
\end{defn}
\begin{rem}
\label{rem:strict independence implies witness independence}($T$
resilient) If $\mathcal{J}$ is strictly independent over $B$ and
$B$ is an extension base, then it is also witness independent over
$B$ by Corollary \ref{cor:witnesses exist}. When $T$ is \NTP{}
the same holds for an $\ind^{i}$-extension base. \end{rem}
\begin{obs}
\label{obs:not witness ind}Let $A=\set{a_{i}}{i<\lambda}$ be a set,
and assume that there is an index $i$ such that $a_{\neq i}\nind^{d}a_{i}$.
 Then $A$ is not witness independent.\end{obs}
\begin{proof}
Let $J$ be a witness starting with $a_{i}$, then (by the assumptions)
there is a formula $\varphi(x,a_{i})$ that divides witnessed by $J$,
and is satisfied by (some finite sub-tuple of) $a_{\neq i}$. Clearly
in this situation $J$ can not be made indiscernible over $a_{\neq i}$. 
\end{proof}
Let us recall the following definition. 
\begin{defn}
We say that the \emph{chain condition}%
\footnote{The terminology is explained in \cite[Remark 2.3]{cheBenya}. %
} holds for a relation $\ind^{*}$ over a set $A$ if: whenever $a\ind_{A}^{*}b$
and $I$ is an infinite $A$-indiscernible sequence starting with
$b$, then there is an indiscernible sequence $I'\equiv_{Ab}I$ such
that $I'$ is indiscernible over $Aa$ and $a\ind_{A}^{*}I$. When
$\ind^{*}=\ind$ we omit ``for $\ind^{*}$ ''. \end{defn}
\begin{fact}
\label{fac:chain condition}\cite{cheBenya} The chain condition holds
for \NTP{} theories over extension bases. \end{fact}
\begin{lem}
\label{lem:chain condition for ind^st}($T$ resilient) The chain
condition holds over extension bases for $\ind^{\st}$.\end{lem}
\begin{proof}
Suppose $a\ind_{A}^{\st}b$ and that $I$ is an $A$-indiscernible
sequence starting with $b$. We must show that there is some indiscernible
sequence $I'\equiv_{Ab}I$ such that $a\ind_{A}^{\st}I'$. By Proposition
\ref{prop:converse to main lemma}, there is some $Ab$-indiscernible
sequence $J$ which is a witness over $A$, starts with $a$, and
such that $J\ind_{A}b$. By Fact \ref{fac:chain condition}, we may
assume that $J\ind_{A}I$ and that $I$ is indiscernible over $AJ$.
By Ramsey and compactness we can find an infinite sequence $J'=\sequence{a_{i}}{i<\omega}$
which is indiscernible over $AI$ and such that:
\begin{itemize}
\item $J'\ind_{A}I$, $I$ is indiscernible over $AJ'$ and $a_{i}b\equiv_{A}ab$
for all $i<\omega$. 
\end{itemize}
By Lemma \ref{lem:main lemma}, it follows that $a_{0}\ind_{A}^{\st}I$.
Applying an automorphism fixing $Ab$ and taking $a_{i}$ to $a$,
we are done. 
\end{proof}
The following corollary is an independence theorem for resilient theories
and $\ind^{\st}$, whose proof is exactly the same as the proof of
\cite[Theorem 3.3]{cheBenya}, so we omit it. 
\begin{cor}
($T$ resilient) Suppose that $A$ is an extension base, $b,b'$ have
the same Lascar strong type over $A$, $a\ind_{A}bb'$ and $c\ind_{A}^{\st}ab$.
Then there is some $c'$ such that $c'a\equiv_{A}ca$, $c'b'\equiv_{A}cb$
and $c'\ind_{A}^{\st}ab'$. \end{cor}
\begin{thm}
\label{thm:Main}Let $a,b$ be tuples and $A$ a set. Consider the
conditions (1)--(5) below. 

In general, (2) implies (3). 

When $T$ is \NTP{} and $A$ is an extension base, (4) implies (5).
If $A$ is an $\ind^{i}$-extension base then (3) implies (4). 

When $T$ is resilient and $A$ is an extension base, (1), (3), (4),
and (5) are equivalent. 

When $T$ is dependent and $A$ is an extension base, (1)--(5) are
equivalent. 
\begin{enumerate}
\item $a\ind_{A}^{\st}b$
\item $a\ind_{A}^{\ist}b$
\item $\left\{ a,b\right\} $ is a strictly independent set over $A$.
\item $\left\{ a,b\right\} $ is a witness independent set over $A$. 
\item $b\ind_{A}^{\st}a$
\end{enumerate}
\end{thm}
\begin{proof}
By Lemma \ref{lem:(Shelah's-Lemma)},\textbf{ }(2) $\Rightarrow$
(3). 

Now assume that $T$ is \NTP{}. The fact that (3) implies (4) when
$A$ is an $\ind^{i}$-extension base was already observed in Remark
\ref{rem:strict independence implies witness independence}. Now assume
that $A$ is an extension base. Suppose that (4) holds, i.e., that
$\left\{ a,b\right\} $ is a witness independent set over $A$. Let
$I$ be a witness for $a$ over $ $$A$, starting with $a$ and $J$
be a witness for $b$ over $A$, starting with $b$, and assume that
$I$ and $J$ are mutually indiscernible over $A$. By Observation
\ref{obs:direct consequence of Kim}, $J\ind_{A}a$. By Lemma \ref{lem:main lemma},
$b\ind_{A}^{\st}a$. This shows that (4) implies (5). 

Next assume that $T$ is resilient. Remark \ref{rem:strict independence implies witness independence}
shows that (3) implies (4). Assume that $A$ is an extension base
and that $a\ind_{A}^{\st}b$. Suppose we are given two indiscernible
sequences $I$ and $J$, starting with $a$ and $b$ respectively.
By the chain condition for $\ind^{\st}$ (Lemma \ref{lem:chain condition for ind^st}),
we may assume that $a\ind_{A}^{\st}J$ and that $J$ is indiscernible
over $Aa$. Hence $J\ind_{A}a$, and by Fact \ref{fac:dividing},
we may assume that $I$ is indiscernible over $AJ$. But then every
pair $ $$\left(a_{i},b_{j}\right)$ for $i,j<\omega$ has the same
type over $A$. By Claim \ref{cla:making mutually indiscernible}
we can find $I'$ and $J'$ such that $I'\equiv_{Aa}I$, $J'\equiv_{Ab}J$
and $J',I'$ mutually indiscernible. This shows that (1) implies (3).
We already saw that (3) implies (4) which implies (5), so by the symmetric
nature of (3), we get that (1), (3), (4), and (5) are equivalent. 

Next assume that $T$ is dependent and that $A$ is an extension base.
Noting that $\ind^{i}=\ind$ we are done by the previous cases.\end{proof}
\begin{cor}
\emph{\label{cor:Symmetry-of-strict}(Symmetry of strict non-forking)}
Suppose $T$ is resilient. Let $A$ be an extension base. If $b\ind_{A}^{\st}a$
then $a\ind_{A}^{\st}b$.
\end{cor}
For \NTP{} we get:
\begin{cor}
($T$ \NTP{}) If $a\ind_{A}^{\ist}b$ for an $\ind^{i}$-extension
base $A$, then $b\ind^{\st}a$.
\end{cor}

\begin{cor}
\label{cor:approximation to main problem}$ $
\begin{enumerate}
\item ($T$ \NTP{}) Suppose $I=\left\langle a_{i}\left|\, i<\omega\right.\right\rangle $
is an indiscernible witness for $a_{0}=a$ over an extension base
$A$. If $I\ind_{A}b$ and $b\ind_{A}^{i}I$, then $a\ind_{A}^{\st}b$.
\item ($T$ dependent) Suppose $C$ is an extension base and $A$ is a set
that has the property that for any finite $a\subseteq A$ there is
some indiscernible witness over $C$ starting with $a$ in $A$.  Then,
for any $B$, if $A\ind_{C}B$ and $B\ind_{C}A$ then $A\ind_{C}^{\st}B$.
\end{enumerate}
\end{cor}
\begin{proof}
(1) follows from Lemma \ref{lem:main lemma} and by preservation of
indiscernibility. By the finite character of $\ind^{\ist}$ (see Lemma
\ref{lem:stind-eqdef}), (2) follows from (1) as $\ind^{i}=\ind$. \end{proof}
\begin{claim}
\label{cla:kind of local character}($T$ \NTP{}) Suppose $I$ is
a sequence of length $\left(\left|T\right|+\left|A\right|\right)^{+}$
which is a witness over an extension base $A$. Then%
\footnote{Here and in the following corollary, for simplicity of the presentation
we assume that $a$ and $b$ are finite tuples, but one can easily
modify the assumptions to suit infinite tuples.%
} for any $a$ there is some $b\in I$ such that $a\ind_{A}b$.\end{claim}
\begin{proof}
If not, then for every $b\in I$, there is some formula $\varphi_{b}\left(x,y\right)$
over $A$ and such that $\varphi_{b}\left(a,b\right)$ holds and $\varphi_{b}\left(x,b\right)$
divides over $A$. We can find an infinite subset $I_{0}$ of $I$
such that $\varphi_{b}$ is constant for all $b\in I_{0}$. Thus we
have a contradiction to the fact that $I$ is a witness. 
\end{proof}
The next corollary gives us another (without use of the chain condition
for $\ind^{\st}$) path to symmetry:
\begin{cor}
\label{cor:witness on the left}($T$ resilient) Assume that $I$
is a sequence of length $\left(\left|T\right|+\left|A\right|\right)^{+}$
which is a witness over an extension base $A$ and that $I\ind_{A}a$.
Then for some $b\in I$, $a\ind_{A}^{\st}b$. 

It follows that if $b\ind_{A}^{\st}a$, then $a\ind_{A}^{\st}b$. \end{cor}
\begin{proof}
Let $J$ be a countable strict Morley sequence over $A$ starting
with $a$ (see Corollary \ref{cor:witnesses exist}). Since $I\ind_{A}a$,
we may assume that $J$ is indiscernible over $AI$. By Claim \ref{cla:kind of local character}
there is some $b\in I$ such that $J\ind_{A}b$. But in addition,
$J$ is indiscernible over $Ab$. By Lemma \ref{lem:main lemma} we
are done. 

Symmetry now follows: by Proposition \ref{prop:converse to main lemma},
we can find such a sequence $I$ which is indiscernible over $Aa$. \end{proof}
\begin{rem}
($T$ dependent) If in Corollary \ref{cor:witness on the left} we
assume that $I$ is an \uline{indiscernible} witness over $A$,
then it is enough to assume it has length $\left|T\right|^{+}$, because
then by ``shrinking of indiscernibles'' (see \cite{Sh715}), some
end-segment of $I$ is indiscernible over $aA$ and we can use Lemma
\ref{lem:main lemma} and symmetry. 
\end{rem}
We cannot expect full transitivity of $\ind^{\st}$ in general, due
to the following proposition, for which we recall:
\begin{defn}
\label{def:gen.stable}A type $p\in S(A)$ in a dependent theory $T$
is called \emph{generically stable }if some (equivalently, any) Morley
sequence in $p$ is an indiscernible set.\end{defn}
\begin{prop}
\label{prop:transitivity implies gen.stable}($T$ dependent) Suppose
$p$ is a type over an extension base $A$ such that strict non-forking
is transitive on realization of $p$, namely if $B\ind_{A}^{\st}CD$
and $C\ind_{A}^{\st}D$ then $BC\ind_{A}^{\st}D$ where $B$, $C$
and $D$ are sets of realizations of $p$. Then $p$ is generically
stable.\end{prop}
\begin{proof}
By Fact \ref{fac:strict_exist}, there is some global strictly non-forking
(over $A$) type $q$ which extends $p$. Let $M\supseteq A$ be some
model. To show that $p$ is generically stable, it is enough to show
that for every $n$ and every permutation $\sigma:n\to n$, if $a=\left\langle a_{i}\left|\, i<n\right.\right\rangle \models q^{\left(n\right)}|_{M}$
then $a_{\sigma}=\left\langle a_{\sigma\left(i\right)}\left|\, i<n\right.\right\rangle \models q^{\left(n\right)}|_{M}$.
It is enough to show it for $\sigma=\left(i\;\, i+1\right)$ where
$i+1<n$. By induction, we may assume $i=n-2$. 

Let $\left\langle a_{i}\left|\, i<n+1\right.\right\rangle \models q^{\left(n+1\right)}|_{M}$.
Since $a_{n}\ind_{A}^{\st}a_{<n-2}a_{n-2}a_{n-1}M$ , by transitivity
and symmetry, we have $a_{<n-2}a_{n}a_{n-2}M\ind_{A}^{\st}a_{n-1}$.
By symmetry again, we have $a_{n-1}\ind_{A}^{\st}a_{<n-2}a_{n}a_{n-2}M$.
By Fact \ref{fac:strongsplit}, since $a_{n}a_{<n-2}\equiv_{M}a_{n-2}a_{<n-2}$
(they both realize $q^{\left(n-1\right)}|_{M}$), $a_{n}a_{n-1}a_{<n-2}\equiv_{M}a_{n-2}a_{n-1}a_{<n-2}$.
Since the left hand side realizes $q^{\left(n\right)}|_{M}$, we are
done. \end{proof}
\begin{cor}
(poor man's transitivity) Assume $T$ is resilient. Suppose that both
$A$ and $Ab$ are extension bases. If $b\ind_{A}a_{1}$, $b\ind_{A}a_{2}$
and $a_{1}\ind_{Ab}^{\st}a_{2}$ then $a_{1}\ind_{A}^{\st}a_{2}$. 

If $T$ is dependent then we can drop the assumption that $Ab$ is
an extension base. \end{cor}
\begin{proof}
First assume that $T$ is resilient and that both $A$ and $Ab$ are
extension bases. By Theorem \ref{thm:Main}, $\left\{ a_{1},a_{2}\right\} $
is a witness independent set over $Ab$ which means that there are
witnesses $I_{1}=\sequence{a_{i}^{1}}{i<\omega}$ and $I_{2}=\sequence{a_{i}^{2}}{i<\omega}$
such that both are mutually indiscernible over $Ab$ and $a_{1}=a_{0}^{1}$,
$a_{2}=a_{0}^{2}$. We will be done by showing that $I_{1},I_{2}$
are also witnesses over $A$ (which will show that $\left\{ a_{1},a_{2}\right\} $
is a witness independent set over $A$). Indeed, if $\varphi\left(x,a_{1}\right)$
divides over $A$, then since $b\ind_{A}a_{1}$, it follows from Fact
\ref{fac:dividing} that $\varphi\left(x,a_{1}\right)$ divides over
$Ab$, so $I_{1}$ witnesses it. Similarly this is true for $I_{2}$. 

In the NIP case we do not need to assume that $Ab$ is an extension
base. Again, by Theorem \ref{thm:Main}, it is enough to show that
$\left\{ a_{1},a_{2}\right\} $ is a strictly independent set over
$A$. Let $I_{1},I_{2}$ be indiscernible sequences over $A$ starting
with $a_{1},a_{2}$. As $b\ind_{A}a_{1}$, there is some $I_{1}'\equiv_{A}I_{1}$
starting with $a_{1}$ that is indiscernible over $Ab$, and there
is some $I_{2}'\equiv_{A}I_{2}$ starting with $a_{2}$ and indiscernible
over $Ab$. Since $a_{1}\ind_{Ab}^{\st}a_{2}$ we may assume that
$I_{1}',I_{2}'$ are mutually indiscernible over $Ab$ (this does
not require $Ab$ to be an extension base by Lemma \ref{lem:(Shelah's-Lemma)})
and in particular over $A$. 

\medskip{}

\end{proof}
We conclude with a few remarks on simple theories. 
\begin{thm}
\label{thm:simple}Strict non-forking equals forking iff $T$ is simple.\end{thm}
\begin{proof}
Lemma \ref{lem:stind-eqdef} holds in any theory. If $T$ is simple,
then non-forking is symmetric, and hence by that lemma, equals strict
non-forking. On the other hand, if $T$ is not simple, then the proof
of \cite[Proposition 2.3.8]{WagnerBook} gives some $a,b,A$ such
that $a\ind_{A}b$ and $b\nind_{A}^{d}a$, so strict non-forking does
not equal non-forking. 
\end{proof}
Note that by Theorem \ref{thm:simple}, Corollary \ref{cor:strict morley are witnesses in resilient}
is a generalization of Kim's lemma for simple theories. Also note
that in simple theories every set is an extension base.
\begin{cor}
\label{cor:simplesequences}Let $T$ be simple, $I$ an infinite indiscernible
sequence. Then the following are equivalent:
\begin{enumerate}
\item $I$ is a Morley sequence. 
\item $I$ is a witness. 
\item $I$ is a strict Morley sequence. 
\item $I$ is totally strict. 
\end{enumerate}
\end{cor}
\begin{proof}
By Corollaries \ref{cor:strictmorley_char} and \ref{cor:2way}, Theorem
\ref{thm:characterize_witness} and symmetry of non-forking. \end{proof}
\begin{prop}
\label{prop:simple-strictind}Let $T$ be simple, let $A=\set{a_{i}}{i<n}$
be a (forking) independent set over a set $A$, and let $I_{i}$ be
an indiscernible sequence starting with $a_{i}$. Then there exist
sequences $I'_{i}$ such that
\begin{itemize}
\item $I'_{i}\equiv_{Aa_{i}}I_{i}$.
\item The set $\set{I'_{i}}{i<n}$ is independent over $A$.
\item $I'_{i}$ is indiscernible over $I'_{<i}a_{>i}A$.
\end{itemize}
\end{prop}
\begin{proof}
For simplicity assume $A=\emptyset$. We prove by induction on $k\le n$
that there are sequences $I'_{i}$ for $i<k$ such that: 
\begin{enumerate}
\item $I'_{i}\equiv_{a_{i}}I_{i}$.
\item The set $\set{I'_{i}}{i<k}\cup\set{a_{i}}{k\leq i<n}$ is independent.
\item $I'_{i}$ is indiscernible over $I'_{<i}a_{>i}$.
\end{enumerate}
So assume that $I'_{<k}$ satisfying the requirements (1)--(3) already
exist. By (2), $a_{k}\ind I'_{<k}a_{>k}$, hence $I'_{<k}a_{>k}\ind a_{k}$.
By the chain condition, there is $I'_{k}\equiv_{a_{k}}I_{k}$ such
that $I'_{<k}a_{>k}\ind I'_{k}$ and $I'_{k}$ is indiscernible over
$I'_{<k}a_{>k}$. This completes the proof. \end{proof}
\begin{thm}
\label{thm:simple-allequivalent} Let $T$ be simple, $B$ a set of
parameters, and let $A=\set{a_{i}}{i<\lambda}$ be a set. Then the
following are equivalent:
\begin{enumerate}
\item $A$ is $B$-independent.
\item $A$ is strictly independent over $B$. 
\item $A$ is witness independent over $B$. 
\item The sequence $\sequence{a_{i}}{i<\lambda}$ is strictly non-forking
over $B$. 
\end{enumerate}
\end{thm}
\begin{proof}
For simplicity assume $B=\emptyset$. 

The equivalence of (1) and (4) follows from Theorem \ref{thm:simple}. 

(1) implies (2) by Proposition \ref{prop:simple-strictind} (as in
the proof of Lemma \ref{lem:(Shelah's-Lemma)}). %

(2) implies (3) by Remark \ref{rem:strict independence implies witness independence}
(since every set is an extension base).

(3) implies (1): if $A$ is dependent, by symmetry there is an index
$i$ such that $a_{\neq i}\nind a_{i}$. Since dividing equal forking,
$A$ is not witness independent by Observation \ref{obs:not witness ind}.

\end{proof}

\section{Quantitative results}

This section contains quantitative analogues of several results from
the previous sections, assuming that the theory (or just a certain
formula) is dependent. The proofs are somewhat technical, and these
results are not really used in the rest of the paper, so this section
can be omitted in the first reading. 
\begin{defn}
\label{def:bounded alt-rank}A theory has \emph{bounded alternation
rank} if for every $n<\omega$ there is some $m<\omega$ such that
whenever $\varphi\left(x,y\right)$ is a formula with $\lg\left(x\right)=n$,
the alternation rank of $\varphi\left(x,y\right)$ is less than $m$,
i.e., there is no indiscernible sequence $\left\langle b_{i}\left|\, i<\omega\right.\right\rangle $
such that $\left\{ \varphi\left(x,b_{i}\right)^{i\pmod2}\left|\, i<m+1\right.\right\} $
is consistent. 
\end{defn}
Note that if $T$ has bounded alternation rank, then it is dependent.

The first result may be considered as a quantitative version of Proposition
\ref{prop:witness alternation rank}.
\begin{prop}
\label{prop:witness alternation rank-1} Let $A$ be any set. Suppose
$\varphi\left(x,y\right)\in L\left(A\right)$ has alternation rank
less than $m$, and let $n=\mplus$. If $I=\left\langle a_{i}\left|\, i<\omega\right.\right\rangle $
is a Morley sequence over $A$ such that $a_{<n}\ind_{A}a_{n}$ and
$\varphi\left(x,a_{0}\right)$ divides over $A$ then $\left\{ \varphi\left(x,a_{i}\right)\left|\, i<n\right.\right\} $
is inconsistent.\end{prop}
\begin{proof}
Suppose not. It is enough to show that 
\[
\left\{ \neg\varphi\left(x,a_{i}\right)\left|\, i\in E\right.\right\} \cup\left\{ \varphi\left(x,a_{i}\right)\left|\, i\in O\right.\right\} 
\]
 is consistent where $O$ ($E$) is the set of odd (even) numbers
smaller than $m+1$. 

If not, then, letting $\Sigma_{1}=\left\{ \varphi\left(x,a_{i}\right)\left|\, i\in E\right.\right\} $
and $\Sigma_{2}=\left\{ \varphi\left(x,a_{i}\right)\left|\, i\in O\right.\right\} $,
we get that $\Sigma_{2}\models\bigvee\Sigma_{1}$. Let $\Sigma_{0}\subseteq\Sigma_{1}$
be minimal such that $\Sigma_{2}\models\bigvee\Sigma_{0}$. Assume
that $\Sigma_{0}$ is not empty, and let $i_{0}\in E$ be minimal
such that $\varphi\left(x,a_{i_{0}}\right)\in\Sigma_{0}$. Let $J=O\cup E_{>i_{0}}$.
By assumption, $\left\{ a_{i}\left|\, i\in O_{<i_{0}}\right.\right\} \ind_{A}a_{i_{0}}$,
and since $I$ is a Morley sequence, by transitivity we have $\left\{ a_{i}\left|\, i\in J\right.\right\} \ind_{A}a_{i_{0}}$. 

Since $\varphi\left(x,a_{i_{0}}\right)$ divides over $A$, it also
divides over $A\cup\left\{ a_{i}\left|\, i\in J\right.\right\} $
(Claim \ref{cla:preservationdividing}), so there is an indiscernible
sequence $\left\langle b_{j}\left|\, j<\omega\right.\right\rangle $
with $b_{0}=a_{i_{0}}$ that witnesses this. By indiscernibility,
for all $j<\omega$ we have 

\[
\Sigma_{2}\vdash\varphi(x,b_{j})\lor\bigvee\Sigma_{0}\backslash\left\{ \varphi\left(x,a_{i_{0}}\right)\right\} .
\]
But then $\Sigma_{2}\vdash\bigvee\left(\Sigma_{0}\backslash\left\{ \varphi\left(x,a_{i_{0}}\right)\right\} \right)$,
contradicting the minimality of $\Sigma_{0}$. 

So $\Sigma_{0}$ is empty, i.e., $\Sigma_{2}\vdash\bot$, contradicting
our assumption. \end{proof}
\begin{rem}
One can extract another argument for Proposition \ref{prop:witness alternation rank-1}
from \cite[proof of Claim 5.19]{Sh783}, assuming that $T$ is extensible:

Suppose $c\vDash\left\{ \varphi\left(x,a_{i}\right)\left|\, i<n\right.\right\} $,
and let $c'$ be such that $c'\equiv_{Aa_{O}}c$ and $c'\ind_{Aa_{O}}I$,
where $a_{O}=\left\{ a_{i}\left|\, i\in O\right.\right\} $. In particular
$\varphi\left(c',a_{i}\right)$ holds for $i\in O$. If $i\in E$,
then by transitivity and the assumption, $c'a_{O}\ind_{A}a_{i}$,
so $\neg\varphi\left(c',a_{i}\right)$ holds. \end{rem}
\begin{cor}
\label{cor:dp-minimal witness}Suppose $T$ has bounded alternation
rank, and $A$ is an extension base. Suppose that $\varphi\left(x,y\right)\in L\left(A\right)$,
$x$ is a tuple of length $n$, and let $m$ be as in Definition \ref{def:bounded alt-rank}.
Suppose $\left\langle a_{i}\left|\, i<\omega\right.\right\rangle $
is a Morley sequence such that $a_{<k}\ind_{A}a_{k}$ where $k=\mplus$.
Then, if $\varphi\left(x,a_{0}\right)$ forks (divides) over $A$,
then $\left\{ \varphi\left(x,a_{i}\right)\left|\, i<k\right.\right\} $
is inconsistent. 

If $T$ is dp-minimal (see e.g., \cite{AlfAlex-dp,Simon-Dp-min}),
then by \cite{AlexAlfTemp}, it has bounded alternation rank. In fact,
$m=2n+1$, so $k=n+1$. \end{cor}
\begin{rem}
In a general dependent theory, it is known by \cite{HP} that if $a,b\models p\in S\left(\C\right)$
and $p$ does not fork over $A$ then $a$ and $b$ have the same
Lascar strong type iff there is an indiscernible sequence $\bar{c}$
over $A$ such that $a\bar{c}$ and $b\bar{c}$ are both indiscernible
over $A$. In fact, if $a\ind_{A}b$ and $a$ and $b$ have the same
Lascar strong type over $A$, then $\left\langle b,a\right\rangle $
starts a Morley sequence over $A$. Indeed, let $M$ be some model
containing $A$ such that $a\ind_{A}bM$, and let $p\in S\left(\C\right)$
be a global non-forking extension of $\tp\left(a/Mb\right)$. So $p$
is invariant over $M$, and both $\bar{a}=\left\langle a_{i}\left|\, i<\omega\right.\right\rangle \models p^{\left(\omega\right)}|_{Mab}$
and $a\frown\bar{a}$ are indiscernible over $M$. Since $p^{\left(\omega\right)}$
does not fork over $A$, it follows from Fact \ref{fac:strongsplit}
that $a\bar{a}\equiv_{A}b\bar{a}$, so both are indiscernible over
$A$. But $a_{0}a_{1}\equiv_{M}aa_{1}\equiv_{A}ba_{1}\equiv_{M}ba$
and we are done. 
\end{rem}
From this we get:
\begin{cor}
In the context of Corollary \ref{cor:dp-minimal witness}, if $T$
is dp-minimal, $n=1$, $a_{0}\ind_{A}a_{1}$ and $a_{1}\ind_{A}a_{0}$,
$\varphi\left(x,a_{0}\right)$ forks over $A$, and $a_{0}$ and $a_{1}$
have the same Lascar strong type over $A$, then $\left\{ \varphi\left(x,a_{0}\right),\varphi\left(x,a_{1}\right)\right\} $
is inconsistent. 
\end{cor}
The next proposition a strengthening of Lemma \ref{lem:main lemma}
under the assumption of bounded alternation rank.
\begin{prop}
\label{prop:bounded version of main lemma}Assume that $T$ has bounded
alternation rank, and that $A$ is an extension base. Assume that
$a$ is a tuple of length $n$ and let $m$ be the corresponding number
from Definition \ref{def:bounded alt-rank}. Then, if $b=\left\langle b_{i}\left|\, i<\mplus\right.\right\rangle $
starts an infinite indiscernible witness over $A$, $b_{i}\equiv_{Aa}b_{0}$
for all $i<\mplus$ and $b\ind_{A}a$ then $b_{0}\ind_{A}^{\st}a$. 
\end{prop}
Note that we cannot use the proof of Lemma \ref{lem:main lemma} directly.
\begin{proof}
By Lemma \ref{lem:main lemma} (and Remark \ref{rem:weakining}),
it is enough to show that there is a witness $I$ over $A$ (not necessarily
indiscernible), starting with $b_{0}$, with all tuples in it having
the same type over $Aa$, such that $I\ind_{A}a$. So let $I=\left\langle b_{i}\left|\, i<\omega\right.\right\rangle $
be an indiscernible witness starting with $b$ over $A$. Let $p\left(\bar{x}\right)=\tp\left(I/A\right)$,
and $r\left(x\right)=\tp\left(b_{0}/Aa\right)$. What we need to show
consistent is
\[
p\left(\bar{x}\right)\cup\left\{ \neg\psi\left(\bar{x},a\right)\left|\,\psi\left(\bar{x},y\right)\in L\left(A\right),\psi\left(\bar{x},a\right)\mbox{ divides over }A\right.\right\} \cup\bigcup_{i<\omega}r\left(x_{i}\right).
\]
Suppose it is not consistent. Then there is some formula $\varphi\left(x,y\right)$
over $A$ and $k<\omega$ such that $\varphi\left(x,a\right)\in r\left(x\right)$
and 
\[
p\left(\bar{x}\right)\cup\left\{ \neg\psi\left(\bar{x},a\right)\left|\,\psi\left(\bar{x},y\right)\in L\left(A\right),\psi\left(\bar{}x,a\right)\mbox{ divides over }A\right.\right\} \vdash\bigvee_{i<k}\neg\varphi\left(x_{i},a\right).
\]
Let $I'=\left\langle b_{i}\left|\, i\in\mathbb{Q}\right.\right\rangle $
be an $A$-indiscernible sequence containing $I$. As $b\ind_{A}a$,
we can, by left extension, assume that $I'\ind_{A}a$. Note that $\left\langle b_{-k+1+i}\left|\, i<\omega\right.\right\rangle \models p$
and independent from $a$ over $A$, so for some $0<i<k$, $\neg\varphi\left(b_{-i},a\right)$
holds ($i\neq0$ because we know that $\varphi\left(b_{0},a\right)$
holds). Similarly, for every $i<\mplus$ there is some $j_{i}\in\left(i,i+1\right)$
such that $\neg\varphi\left(b_{j_{i}},a\right)$ holds, and there
is such $b_{j_{\infty}}$ for $j_{\infty}>\mplus$. Since $\varphi\left(b_{i},a\right)$
holds for every $0\leq i<\mplus$, this gives a contradiction to the
choice of $m$.\end{proof}
\begin{rem}
If $T$ is dp-minimal, if $n=1$ and $A$ is a model $M$, we can
replace `` $\left\langle b_{i}\left|\, i<2\right.\right\rangle $
starts an indiscernible witness'' by ``$b_{1}\ind_{M}^{\st}b_{0}$''
(because then we can generate a strict Morley sequence starting with
$\left\langle b_{0},b_{1}\right\rangle $). 
\end{rem}

\section{On types related to generically stable types.}

The following notions arises when studying extensions of ``special''
(e.g., generically stable) types:
\begin{defn}
We say that a type $p\in S\left(A\right)$ is \emph{co-dominated}
(over $A$) by a type $q\in S\left(A\right)$ if there are $a,b$
realizing $p,q$ respectively such that whenever $c\ind_{A}b$, we
also have $c\ind_{A}a$. We denote this by $p\triangleleft_{A}^{*}q$
or simply by $p\triangleleft^{*}q$ , when $A$ is clear from the
context. To specify $a$ and $b$ we write $a\triangleleft^{*}b$. 
\end{defn}
Recall that a type $p\in S(A)$ is \emph{dominated} by $q\in S(A)$
($p\triangleleft q$) if there are $a,b$ realizing $p,q$ respectively
such that whenever $b\ind_{A}c$ , we also have $a\ind_{A}c$. It
was shown in \cite{OnUs-stable} that if $T$ is dependent, $q$ is
generically stable and $p\triangleleft q$, then $p$ is also generically
stable. In this section we investigate (using the techniques developed
earlier in the paper) types co-dominated by generically stable types.
\begin{thm}
\label{thm:dombystable-strict}($T$ dependent) Let $A$ be an extension
base. Suppose $p\in S(A)$ is generically stable, $q\triangleleft_{A}^{*}p$
, then: any Morley sequence in $q$ over $A$ is in fact a strict
Morley sequence.\end{thm}
\begin{proof}
To simplify the notation, assume $A=\emptyset$. Suppose $\left\langle b_{i}\left|i<\kappa\right.\right\rangle $
is a Morley sequence in $q$. By Corollary \ref{cor:strictmorley_char}
and Theorem \ref{thm:characterize_witness}, it is enough to show
that $b_{\neq i}\ind b_{i}$ for all $i$.
\begin{claim*}
There exists $\left\langle a_{i}\left|\, i<\kappa\right.\right\rangle $
such that\end{claim*}
\begin{itemize}
\item $a_{i}\models p$,
\item $\left\langle a_{i}b_{i}\left|\, i<\omega\right.\right\rangle $ is
a non-forking sequence,
\item $a_{i}b_{i}\equiv a_{j}b_{j}$ for all $i,j$, and, most importantly, 
\item $b_{i}\triangleleft^{*}a_{i}$.\end{itemize}
\begin{proof}
For $i<\kappa$, we define $I_{i}=\left\langle a_{i}^{j}\left|\, j<i\right.\right\rangle $
that satisfy the conditions for $b_{<i}$, and such that if $i<i'$
then $b_{<i}I_{i}\equiv b_{<i}I_{i'}\upharpoonright i$. 

For $i=0$, there is nothing to define.

For $i=1$, let $a_{1}^{0}\models p$ be such that $b_{0}\triangleleft^{*}a_{1}^{0}$. 

For $i>1$ a limit ordinal, use compactness to find $I_{i}=\left\langle a_{i}^{j}\left|j<i\right.\right\rangle $
such that $a_{i}^{<j}b_{<j}\equiv I_{j}b_{<j}$ for all $j<i$. 

For $i+1$, let $I_{i}'\equiv_{b_{<i}}I_{i}$ be such that $b_{i}\ind b_{<i}I_{i}'$,
and denote $I_{i}'=\sequence{a_{i+1}^{j}}{j<i}$. Now, by left extension
(Corollary \ref{cor:Left Extension}), find some $a_{i+1}^{i}$ such
that $a_{1}^{0}b_{0}\equiv a_{i+1}^{i}b_{i}$ and $a_{i+1}^{i}b_{i}\ind I_{i}'b_{<i}$.
So letting $I_{i+1}=\left\langle a_{i+1}^{j}\left|\, j<i\right.\right\rangle $,
we are done.

Now use compactness to finish as in the limit case.
\end{proof}
Letting $\kappa$ be big enough, we can use Erd\H{o}s-Rado (as in
the proof of Proposition \ref{prop:converse to main lemma}) to find
one which is also indiscernible (of course, the sequence $\langle b_{i}\rangle$
will now change, but we are keeping its type, which is all that is
important).

Now assume that the order type of the sequence $\langle a_{i}b_{i}\rangle$
is $\left(\mathbb{Q},<\right)$. 

Note that $I=\left\langle a_{i}\left|\, i\in\mathbb{Q}\right.\right\rangle $
is a Morley sequence in $p$, so (since $p$ is generically stable)
it is an indiscernible set.

Recall that we are trying to show that $b_{\neq i}\ind b_{i}$ for
all $i$. Assume that $b_{\neq i}\nind b_{i}$, then by co-dominance,
$b_{\neq i}\nind a_{i}$. By symmetry for generically stable types
\cite[Lemma 8.5]{Us}, $a_{i}\nind b_{\neq i}$. Hence there is a
formula $\varphi\left(a_{i},b_{\neq i}\right)$ that shows it (i.e.,
$\varphi\left(x,b_{\neq i}\right)$ divides). Write this formula as
$\varphi\left(a_{i},b_{<i-\varepsilon},b_{>i+\varepsilon}\right)$
for some $\varepsilon\in\mathbb{Q}$ small enough. By indiscernibility,
$\varphi\left(a_{j},b_{<i-\varepsilon},b_{>i+\varepsilon}\right)$
holds for all $j\in\left(i-\varepsilon,i+\varepsilon\right)$. Since
$I$ is an indiscernible set, and $T$ is dependent, $\varphi\left(a_{j},b_{<i-\varepsilon},b_{>i+\varepsilon}\right)$
for almost all (i.e., except finitely many) $j\in\mathbb{Q}$. But
$b_{>i+\varepsilon}$ is bounded, i.e., it is in fact contained in
$b_{<i'}$ for some $i'\in\mathbb{Q}$. So for some $j>i'$, $a_{j}\nind b_{<i'}$
--- a contradiction. \end{proof}
\begin{cor}
\label{cor:dombystable-symm}($T$ dependent) Suppose $A$ is an extension
base, $p,q\in S\left(A\right)$, $q\triangleleft^{*}p$ and $p$ is
generically stable.

Then: for any $b$, if $a\ind_{A}b$ and $a\models q$ then $a\ind_{A}^{\st}b$.
In particular, forking is symmetric for realizations of $q$ {[}caution!
not necessarily for \emph{tuples} of realizations{]}.\end{cor}
\begin{proof}
Suppose $a\ind_{A}b$. Then let $p$ be a global non-forking type
extending $\tp\left(a/Ab\right)$. Generate a Morley sequence $I$
starting with $a$ using $p$, so $I$ is indiscernible over $Ab$,
but also $I\ind b$. As $I$ is a strict Morley sequence by the previous
theorem, it follows that $a\ind^{\st}b$ by Lemma \ref{lem:main lemma}. \end{proof}
\begin{rem}
Suppose that for any $b$, if $b\ind_{A}a$ and $a\models q$ then
$a\ind_{A}b$. Then $q$ is generically stable: 

Let $\left\langle a_{i}\left|\, i<\omega\right.\right\rangle $ be
a Morley sequence in $q$. Then by transitivity and symmetry, $a_{1}\ind_{A}a_{0}a_{2}$,
so $a_{1}a_{0}\equiv_{A}a_{1}a_{2}\equiv_{A}a_{0}a_{1}$. A similar
argument shows that this is an indiscernible set. This is similar
to the proof of Proposition \ref{prop:transitivity implies gen.stable}. \end{rem}
\begin{example}
Let $T$ be the theory of a dense linear order, and let $q$ be the
type of a singleton over the empty set. By Example \ref{exa:rationals},
$q$ satisfies the conclusions of both Theorem \ref{thm:dombystable-strict}
and Corollary \ref{cor:dombystable-symm}. So none of these conditions
imply generic stability. 
\end{example}

\begin{example}
(due to Pierre Simon) Let $L=\left\{ <,E\right\} $ and $T$ be the
theory of a dense linear order with an equivalence relation with dense
classes. Let $p$ be the type in $T^{\eq}$ over $\emptyset$ of any
$E$-class. Then $p$ is generically stable. Let $q$ be the type
of any element in the home sort. Then $q\triangleleft^{*}p$ as witnessed
by taking $a$ and $a/E$ for some $a$. But $q$ is not generically
stable. 
\end{example}

\begin{example}
It is possible to modify this example to have that $q$ is itself
an extension of a generically stable type. The idea is to add another
equivalence relation $E'$ which is a coarsening of $E$, and the
order relation only applies to $E'$-equivalent elements. Then, the
type of any element $a$ in the home sort is generically stable, but
if $a'E'a$ then $\tp\left(a'/a\right)$ is no longer generically
stable, but it is still co-dominated by the type of its $E$-class
in $T^{\eq}$ over $a$.
\end{example}

\begin{example}
One can even construct an example where $q$ is a forking (not generically
stable) extension of a generically stable type $p'$, such that $q$
is co-dominated by the non-forking (hence generically stable) extension
of $p'$. Let $L=\left\{ <,E_{1},E_{2}\right\} $ and let $T$ be
the model completion of the following universal theory:
\begin{itemize}
\item $E_{1}$ and $E_{2}$ are equivalence relations. 
\item $<$ is a partial order.
\item If $x<y$ then $xE_{1}y$, and $<$ is a linear order on each $E_{1}$-class. 
\end{itemize}
This theory has the amalgamation property and the joint embedding
property, hence the model completion exists and has elimination of
quantifiers (see \cite[Theorem 7.4.1]{Hod}). Note that every set
is an extension base. Let $M\models T$ and $a\in M$. Then it is
easy to see that $\tp\left(a/\emptyset\right)$ is generically stable.
Let $b$ realize the (unique) non-forking extension (hence generically
stable by \cite[Corollary 4.11]{Us}) of $\tp\left(a/\emptyset\right)$
over $a$, and let $p=\tp\left(b/a\right)$. Let $a'$ be such that
$a'\mathrela{E_{1}}a$ and $a'\mathrela{E_{2}}b$ (so $\neg\left(a'\mathrela{E_{1}}b\right)$
and $\neg\left(a'\mathrela{E_{2}}a\right)$). Let $q=\tp\left(a'/a\right)$.
Then $q$ is not generically stable, but $a'\triangleleft^{*}b$:

Suppose $c\ind_{a}b$, and we need to show that $c\ind_{a}a'$. By
quantifier elimination we may assume that $c$ is a singleton. There
is only one possible reason that $c\nind_{a}a'$: $c\mathrela{E_{2}}a'$.
But since $a'\mathrela{E_{2}}b$ and $c\ind_{a}b$ this does not hold. 
\end{example}

\section{\label{sec:My-weight-problem}Strict notions of weight }

During a talk by the second author on an early version of this paper,
Anand Pillay asked whether there is a notion of weight, based on notions
of forking discussed here, that characterizes strong dependence. In
this section we confirm that this is indeed the case. Furthermore,
we isolate notions of weight based on strict independence that characterize
dependence, strong dependence, and the tree property of the second
kind.

Recall:
\begin{defn}
A theory is \emph{strongly dependent} if there is no sequence of formulas
$\left\langle \varphi_{i}\left(x,y\right)\left|\, i<\omega\right.\right\rangle $
and an array $\left\langle a_{i,j}\left|\, i,j<\omega\right.\right\rangle $
such that for every $\eta:\omega\to\omega$, the following set is
consistent 
\[
\left\{ \varphi\left(x,a_{i,j}\right)^{\eta\left(i\right)=j}\left|\, i,j<\omega\right.\right\} .
\]
 Where $\varphi^{1}=\varphi$ and $\varphi^{0}=\neg\varphi$. 
\end{defn}
We will need the following lemma:
\begin{lem}
\label{lem:A invariant is M strictly invariant}If $p$ is a global
$A$ invariant type and $M\supseteq A$ is $\left|A\right|^{+}$ saturated,
then $p$ is an heir over $M$. In this case, $p$ is strictly invariant
over $M$. In particular, a Morley sequence that $p$ generates over
$M$ is totally strict. \end{lem}
\begin{proof}
Suppose $\varphi\left(x,c,m\right)\in p$ for $m\in M$. Let $c'\equiv_{Am}c$
be in $M$. Then $\varphi\left(x,c',m\right)\in p$ by invariance
over $A$, and we are done by Fact \ref{fac:strict_exist}.
\end{proof}
In the following definition, $\ind$ is either $\ind^{f}$, $\ind^{d}$,
$\ind^{i}$ or $\ind^{s}$. 
\begin{defn}
We say that $p\in S\left(B\right)$ has $\ind$-pre-weight at least
$\alpha$, if there is a $B$-strictly independent set $\left\{ a_{i}\left|\, i<\alpha\right.\right\} $
and $a\models p$ such that $a\nind_{B}a_{i}$ for all $i$. So we
say that it has pre-weight $<\alpha$ if it is not the case that it
has pre-weight at least $\alpha$. \end{defn}
\begin{thm}
$ $
\begin{enumerate}
\item $T$ is \NTP{} iff all types have $\ind^{d}$-pre-weight $<\left|T\right|^{+}$
iff all types over models have $\ind^{f}$ / $\ind^{d}$-pre-weight
$<\left|T\right|^{+}$.
\item $T$ is dependent iff all types have $\ind^{s}$-pre-weight $<\left|T\right|^{+}$
iff all types over models have $\ind^{i}$ /\textup{ $\ind^{s}$-}pre-weight
$<\left|T\right|^{+}$.
\item $T$ is strongly dependent iff all types have $\ind^{s}$-pre-weight
$<\aleph_{0}$ iff all types over models have $\ind^{i}$ / $\ind^{s}$-pre-weight
$<\aleph_{0}$.
\end{enumerate}
\end{thm}
\begin{proof}
(1) Let $T$ is \NTP{}, $p\in S\left(B\right)$. If it has $\ind^{d}$-pre-weight
$\geq\left|T\right|^{+}$, then there is a $B$-strictly independent
set $\left\{ a_{i}\left|\, i<\left|T\right|^{+}\right.\right\} $
and $a\models p$ such that $a\nind_{B}^{d}a_{i}$ for all $i$. So
for each $i$ there is a formula $\varphi_{i}\left(x,y\right)$ over
$B$ such that $\varphi_{i}\left(x,a_{i}\right)$ divides over $B$.
We may assume that for $i<\omega$, $\varphi_{i}=\varphi$, and this
is a contradiction to Lemma \ref{lem:(Kim's-Lemma)}. Since over models,
forking equals dividing, we are done. 

On the other hand, suppose $T$ is not \NTP{}, i.e., it has the tree
property of the second kind for some formula $\varphi\left(x,y\right)$
and $k<\omega$. So there is an array $\left\langle a_{i,j}\left|\, i,j<\omega\right.\right\rangle $,
such that every row is $k$ inconsistent and each vertical path is
consistent. We may assume by Ramsey and compactness (see the proof
of Claim \ref{cla:making mutually indiscernible}) that the rows are
mutually indiscernible and that the depth of the array is $\left|T\right|^{+}$.
For $i<\left|T\right|^{+}$, let $A_{i}=\left\{ a_{i,j}\left|\, j<\omega\right.\right\} $,
and $p_{i}$ a global co-heir over $A_{i}$ containing $\left\{ x\neq a_{i,j}\left|\, j<\omega\right.\right\} $.
Let $A=\bigcup_{i<\left|T\right|^{+}}A_{i}$ and $M\supseteq A$ some
$\left|A\right|^{+}$ saturated model. 

Let $b_{i,j}$ be generated by $b_{i,j}\models p_{i}|_{Mb_{i,<j}b_{<i}}$.
So we have an array $\left\langle b_{i,j}\left|\, i<\left|T\right|^{+},j<\omega\right.\right\rangle $.
It is easy to see that $\left\langle b_{i,j}\left|\, i<\left|T\right|^{+},j<\omega\right.\right\rangle $
still witnesses the tree property of the second kind for $\varphi$.
The first column $\left\langle b_{i,0}\left|\, i<\left|T\right|^{+}\right.\right\rangle $
is a strictly invariant sequence over $M$, by choice of $p_{i}$
and Lemma \ref{lem:A invariant is M strictly invariant}. Let $a\models\left\{ \varphi\left(x,b_{i,0}\right)\left|\, i<\left|T\right|^{+}\right.\right\} $,
so $p=\tp\left(a/M\right)$ has $\ind^{d}$-pre-weight at least $\left|T\right|^{+}$
and we are done.

Note also, that in the proof the sequence of infinite tuples $\left\langle \left\langle b_{i,j}\left|\, j<\omega\right.\right\rangle \left|\, i<\left|T\right|^{+}\right.\right\rangle $
is itself strictly invariant over $M$, because $p_{i}^{\left(\omega\right)}$
is $A$-invariant. Since each such tuple is an indiscernible sequence
this means that $p$ has $\ind^{s}$-pre-weight at least $\left|T\right|^{+}$
as well. 

(2) Suppose $T$ is dependent. So it is also \NTP{}. By Fact \ref{fac:strongsplit},
and since forking = dividing over models, (1) implies left to right.

Conversely, assume all types over models have $\ind^{s}$-pre-weight
$<\left|T\right|^{+}$. If $T$ is not dependent then there is a formula
$\varphi\left(x,y\right)$ and an array of mutually indiscernible
sequences $\left\langle a_{i,j}\left|\, i<\left|T\right|^{+},\, j<\omega\right.\right\rangle $
such that the set $\left\{ \varphi\left(x,a_{i,0}\right)\land\neg\varphi\left(x,a_{i,1}\right)\left|\, i<\left|T\right|^{+}\right.\right\} $
is consistent. Now apply the same proof as in (1), and note that since
$p_{i}^{\left(2\right)}$ is also $A$ invariant, the sequence of
pairs $\left\langle b_{i,0},b_{i,1}\right\rangle $ is strictly invariant. 

(3) Suppose $T$ is strongly dependent. If $p\in S\left(B\right)$
has $\ind^{s}$-pre-weight at least $\aleph_{0}$, then we can construct
an array of mutually indiscernible sequences over $B$, $\left\langle a_{i,j}\left|\, i,j<\omega\right.\right\rangle $
and such that $\left\{ \varphi_{i}\left(x,a_{i,0}\right)\land\neg\varphi_{i}\left(x,a_{i,1}\right)\left|\, i<\omega\right.\right\} $
is consistent:

There is some $a\models p$ and a strictly independent set of pairs
$\left\{ \left(a_{i,0},a_{i,1}\right)\left|\, i<\omega\right.\right\} $
over $B$ such that each pair starts an indiscernible sequence over
$B$ and $\varphi\left(a,a_{i,0}\right)\land\neg\varphi\left(a,a_{i,1}\right)$
holds. By definition of strict independence, there are $B$-mutually
indiscernible sequences $I_{i}=\left\langle \left(b_{j}^{i,0},b_{j}^{i,1}\right)\left|\, j<\omega\right.\right\rangle $
(of pairs) starting with $\left(a_{i,0},a_{i,1}\right)$. By dependence,
$\left\langle \ldots b_{j}^{i,0},b_{j}^{i,1}\ldots\right\rangle $
are also mutually indiscernible over $B$ (if not, one can easily
find infinite alternation). 

Since $T$ is dependent this is a contradiction to strong dependence
(for $\varphi_{i}$ in the definition, take $\psi_{i}\left(x,y,z\right)=\varphi_{i}\left(x,y\right)\land\neg\varphi_{i}\left(x,z\right)$).
By Fact \ref{fac:strongsplit}, and since forking = dividing over
models, left to right holds. 

Conversely, assume all types over models have $\ind^{s}$-pre-weight
$<\aleph_{0}$, and proceed as in (2). \end{proof}
\begin{rem}
One can change the definition of weight so that instead of $\left\{ a_{i}\right\} $
being a strictly independent set, it would be a strictly invariant
sequence. The theorem would go through with essentially the same proof. 
\end{rem}

\begin{rem}
Clause (1) in the proposition above, as well as its proof, are similar
to \cite[Theorem 4.9]{ChernikovNTP2}, but were done independently. 
\end{rem}

\section{\label{sec:Problems}Problems}
\begin{problem}
This is one of the main problems that motivated our paper: Is it the
case that (under reasonable assumptions, e.g., resilience, or even
dependence) given a model $M$ (or just an extension base) and tuples
$a,b$, do we have $a\ind_{M}^{\st}b$ if and only if {[}$a\ind_{M}b$
and $b\ind_{M}a${]}?
\end{problem}

\begin{problem}
Are all witnesses strictly independent? This seems a bit too strong,
but perhaps with some more assumptions it becomes true. Note that
by Lemma \ref{lem:main lemma} and Theorem \ref{thm:characterize_witness}
(and the remark following), if $\left\langle a_{i}\right\rangle $
is an indiscernible witness over an extension base $A$ and $T$ is
\NTP{}, then for any $i\neq j$, $a_{i}\ind_{A}^{\st}a_{j}$, so
if $T$ is resilient then (by Theorem \ref{thm:Main}) any pair is
strictly independent. 
\end{problem}

\begin{problem}
Is strict independence equivalent to witness independence (and hence
to strict non-forking) for a pair of tuples over an extension base
in an \NTP{} theory?
\end{problem}

\begin{problem}
Is the original ``Shelah's Lemma'' true in resilient theories? That
is, are strictly non-forking sequences strictly independent? Note
that this is true for both dependent and simple (Theorem \ref{thm:simple-allequivalent})
theories (a priori for quite different reasons). 
\end{problem}

\begin{problem}
Theorem \ref{thm:Main} says that over extension bases strict independence
for two elements (tuples) is the same as strict non-forking (under
resilience). It is not so clear what could be the appropriate analogue
of this results for an arbitrary set of tuples. 
\end{problem}

\begin{problem}
Is strict non-forking symmetric in \NTP{} theories? 
\end{problem}
\bibliographystyle{alpha}
\bibliography{common}

\end{document}